\documentclass{amsart}

      \title[Betti categories of graded modules,    applications to                     toric rings]
            {Betti categories of graded modules and applications to monomial ideals and toric rings}
     \author{Alexandre Tchernev}
     \author{Marco Varisco}
    \address{Department of Mathematics and Statistics, University at Albany, SUNY, USA\medskip}
      \email{\hemail{atchernev@albany.edu}}
    \urladdr{\hurl[]{math.albany.edu/~tchernev/}\medskip}
      \email{\hemail{mvarisco@albany.edu}}
    \urladdr{\hurl{albany.edu/~mv312143/}\medskip}
       \date{May 31, 2016}
\usepackage{amssymb}
\usepackage{aliascnt}
\usepackage{mathtools}
\usepackage{etoolbox}
\usepackage{mfirstuc}
\usepackage{tikz}       \usetikzlibrary{cd}
\usepackage[pagebackref, pdfinfo={
      Title={Betti categories of graded modules and applications to monomial ideals and toric rings},
     Author={Alexandre Tchernev and Marco Varisco},
   Keywords={},
    Subject={2010 MSC: 13D02, 14M25, 05E40, 18G10},
}]{hyperref}
\usepackage[alphabetic, backrefs, msc-links, nobysame]{amsrefs}

\newcommand*{\hurl}  [2][www.]{\href{http://#1#2}{\nolinkurl{#2}}}
\newcommand*{\hemail}[1]{\href{mailto:#1}{\nolinkurl{#1}}}
\newcommand*{\DOI}   [1]{\href{http://dx.doi.org/#1}{\nolinkurl{#1}}}

\newcommand*{\MSC}   [1]{\href{http://www.ams.org/msc/msc2010.html?t=#1}{#1}}

\numberwithin{equation}{section}

\newcommand*{\definetheorem}[3][equation]{%
  \newaliascnt{#2}{#1}
  \newtheorem{#2}[#2]{#3}
  \aliascntresetthe{#2}
  \expandafter\def\csname #2autorefname\endcsname{#3}
}
\makeatletter
\newcommand*{\definetheoremsame}[2][equation]{%
  \definetheorem[#1]{\zap@space#2 \@empty}{\capitalisewords{#2}}
}
\makeatother

\theoremstyle{plain}
\forcsvlist{\definetheoremsame}%
{corollary, lemma, Nakayama's lemma, proposition, theorem}

\theoremstyle{definition}
\forcsvlist{\definetheoremsame}%
{definition, example, remark}

\DeclareMathAlphabet{\matheurm}{U}{eur}{m}{n}

\newcommand*{\define}[5]{%
  \ifstrequal{#2}{*}{\expandafter#1\expandafter*}{\expandafter#1}%
  \csname#4#5\endcsname{#3{#5}}
}
\forcsvlist{\define{\newcommand}{*}{\mathcal}{C}}{%
  A,B,C,D,E,F,G,H,I,J,K,L,M,N,O,P,Q,R,S,T,U,V,W,X,Y,Z,%
}
\forcsvlist{\define{\newcommand}{*}{\mathbb}{I}}{%
  A,B,C,D,E,F,G,H,I,  K,L,M,N,O,P,Q,R,S,T,U,V,W,X,Y,Z%
}                % \IJ already defined
\renewcommand*{\IJ}{\mathbb{J}}

\forcsvlist{\define{\newcommand}{*}{\matheurm}{}}%
{Fun, gr, Mod, Sets}

\forcsvlist{\define{\DeclareMathOperator}{}{}{}}%
{aut, cl, coker, const, id, im, image, incl, ind, minsupp, mor, obj, rank, res, supp}

\newcommand*{\TO}  [1][]{\stackrel{#1}{\longrightarrow}}
\newcommand*{\FROM}[1][]{\stackrel{#1}{\longleftarrow}}
\newcommand*{\MOR} [4][]{#2\colon#3\TO[#1]#4}

\newcommand*{\AND}{\qquad\text{and}\qquad}
\newcommand*{\ANDalign}[3][0pt]{\hskip#1\mathllap{#2}\AND\mathrlap{#3}}

\DeclarePairedDelimiterX\SET[2]{\{}{\}}{\,#1\;\delimsize\vert\;#2\,}

\newcommand*{\D}{\text{-}} % dash in math displays

\DeclareMathOperator*{\tensor}{\otimes}

\newcommand*{\op}{{\operatorname{op}}}

\newcommand*{\gf}{\Bbbk}

\newcommand*{\modules}  [1]{#1\D\Mod}
\newcommand*{\grmodules}[2]{#1\D\gr\,\modules{#2}}

\newcommand*{\cat}{\CC}
\newcommand*{\catwo}{\CD}

\newcommand*{\Betti}{\CB}
\newcommand*{\lubcat}{\CL}

\newcommand*{\iso}[1][\cat]{\matheurm{iso}#1}

\newcommand*{\cla}[1][\cat]{\cl(\obj#1)}

\newcommand*{\jrad}  {\mathfrak{J}}
\newcommand*{\spl}   {\mathfrak{S}}
\newcommand*{\free}  {\IF}
\newcommand*{\forget}{\IU}

\newcommand*{\act}[2]{#1\!\smallint\!#2}

\newcommand*{\cx}[1]{#1_\bullet}

\DeclareMathOperator{\HH}{H}

\newcommand*{\upset}[1]{\operatorname{\uparrow}#1}

\begin{document}

\begin{abstract}
We introduce the notion of Betti category for graded modules over suitably graded polynomial rings, and more generally for modules over certain small categories.
Our categorical approach allows us to treat simultaneously many important cases, such as monomial ideals and toric rings.
We prove that in these cases the Betti category is a finite combinatorial object that completely determines the structure of the minimal free resolution.
For monomial ideals, the Betti category is the same as the Betti poset that we studied in~\cite{bettiposets}.
We describe in detail and with examples how the theory applies to the toric case, and provide an analog for toric rings of the lcm-lattice for monomial ideals.
\end{abstract}

%%%%%%%%%%%%%%%%%%%%%%%%%%%%%%%%%%%%%%%%%%%%%%%%%%%%%%%%%%%%%%%%%%%%%%%%%%%%%%

\maketitle
\tableofcontents

%%%%%%%%%%%%%%%%%%%%%%%%%%%%%%%%%%%%%%%%%%%%%%%%%%%%%%%%%%%%%%%%%%%%%%%%%%%%%%

\section{Introduction}

Toric rings are the affine coordinate rings of toric varieties. They have many applications to various areas of mathematics and a rich underlying combinatorial structure, and have long been the subject of extensive research in geometry, combinatorics, and algebra;
see for example
\citelist{\cite{Fulton} \cite{Bruns-Gubeladze} \cite{Miller-Sturmfels} \cite{Peeva} \cite{Villarreal}}
and the references there.
An important open problem regarding the homological properties of toric rings is understanding the structure of their minimal free resolutions.
Some of the most successful approaches to this problem (see e.g.~\citelist{\cite{Bayer-Sturmfels} \cite{Peeva-Sturmfels} \cite{Peeva-Sturmfels-2}}) are closely related to methods used to attack the analogous open problem of understanding the structure of minimal free resolutions of monomial ideals.
However, regardless of the apparent similarities, these two open problems have always been treated separately, and there are important results on the monomial side that do not have toric counterparts yet.
For example, the lcm-lattice and the Betti poset of a monomial ideal are shown in~\cite{GPW} and~\cite{bettiposets}, respectively, to be discrete combinatorial objects whose isomorphism classes completely determine (in an implicit but natural way) the structure of the minimal free resolution of the monomial ideal.
% and it would be of significant interest to provide toric analogues for them.
In this paper, we introduce toric analogs of these objects, and prove that in a similar way they determine the minimal free resolutions of toric rings.

We achieve this by establishing a unified approach to the study of minimal free resolutions of monomial, toric, and, more generally, homogeneous ideals in suitably graded (not necessarily commutative) polynomial rings.
All of these are seen to be special cases of minimal projective resolutions of modules over certain small categories $\cat$ where all {E}ndomorphisms are {I}somorphisms (hence called EI categories),
which seem to provide the correct framework where graded syzygies over polynomial rings should be studied.

In general, given a commutative ring~$\gf$ and a small category~$\cat$,
a \emph{$\gf\cat$-module} is just a functor from $\cat$ to the category of $\gf$-modules, and a homomorphism of $\gf\cat$-modules is just a natural transformation of such functors.
Modules over categories have been studied extensively in algebraic topology, in particular to investigate equivariant phenomena; see e.g.~\citelist{\cite{tomDieck}*{Section~I.11} \cite{Lueck}*{Section~9} \cite{kc}*{Section~13}}.
More recently they have found applications to the study of stability phenomena in representation theory \citelist{\cite{Church-Ellenberg-Farb} \cite{Church-Ellenberg-Farb-Nagpal}}, and were first brought into the study of minimal free resolutions of monomial ideals by the authors in~\cite{bettiposets}.
We review the foundations of the theory of modules over categories in Sections~\ref{MOD-CAT} through~\ref{EI}.

We expand our investigations from \cite{bettiposets} by studying the properties of minimal projective resolutions of modules over a small EI category~$\cat$.
We describe a version of Nakayama's Lemma in this context (Lemma~\ref{Nak}), and establish sufficient conditions for the existence of minimal projective or free resolutions in \autoref{min-projres-local} and \autoref{min-projres-global}.

In \autoref{BETTI} we introduce the new notion of the \emph{Betti category} $\Betti(M)$ of a
$\gf\cat$-module $M$.
This is a full subcategory of $\cat$, and we show in our first main result, \autoref{main-technical}, that it is in a certain sense the ``smallest'' discrete combinatorial object that retains all essential homological information about $M$.
When $M$ corresponds to a monomial ideal, the Betti category is exactly the Betti poset, and we recover all results from~\cite{bettiposets}.
For a general graded module over a polynomial ring, the Betti category has as objects the degrees of the basis elements in the minimal graded free resolution of the module, and as morphisms the monomials of appropriate degrees; composition of morphisms is multiplication of monomials.
In particular, in many important cases such as monomial ideals and toric rings, the Betti category is finite and can be readily computed without needing to compute minimal free resolutions.

As a main application, in \autoref{TORIC} we analyze in greater detail the special case of toric rings over a field~$\gf$.
More specifically, let $Q$ be a submonoid of the free abelian monoid $\IN^r$, with $\CM=\{a_1,\dotsc, a_n\}$ the minimal generating set for $Q$.
The monoid ring $T=\gf[Q]$ is called a \emph{toric ring}.
We consider $T$ as a $\IZ^r$-graded module over the polynomial ring $R=\gf[x_1,\dotsc, x_n]$, via the homomorphism of polynomial rings $\MOR{\varphi}{\gf[x_1,\dotsc, x_n]}{\gf[\IN^r]}$ with image $T$ that sends $x_i$ to~$a_i$, and where the $\IZ^r$-grading of~$R$ is given by $\deg x_i = a_i$.
Let $\Betti(T)$ be the Betti category of~$T$.
Then:
\begin{itemize}
\item
We show in \autoref{main-theorem-1} that the bar resolution of the constant functor on $\Betti(T)$ lifts to a canonical (non minimal) finite free $\IZ^r$-graded resolution $\cx{F}(T)$ of\/~$T$ over\/~$R$.
This resolution differs from the hull resolution of\/~$T$ constructed in \cite{Bayer-Sturmfels}.
\medskip

\item
We show in \autoref{main-theorem-2} that the equivalence class of the Betti category~$\Betti(T)$ determines completely the structure of the minimal free resolution of $T$ over $R$, in the following sense:
if $T'$ is another toric ring over $\gf$ whose Betti category $\Betti(T')$ is equivalent to $\Betti(T)$, then the minimal free resolution of $T$ is obtained from the minimal free resolution of $T'$ in a functorial way.
In particular, this shows that the Betti category is the correct toric analogue of the Betti poset of a monomial ideal.
\medskip

\item
We introduce the \emph{lub-category} $\lubcat(Q)$, and we show in \autoref{main-theorem-3} that this is a toric analogue of the lcm-lattice of a monomial ideal.
\medskip
\end{itemize}

We also provide an example of two well-known non-isomorphic toric rings with equivalent (in fact isomorphic) Betti categories.
This raises the interesting question of characterizing those categories that are Betti categories of toric rings over a given field~$\gf$.
In the case of monomial ideals, the analogous question was answered in~\cite{bettiposets}*{Theorem~6.4 on page~5126}.

%%%%%%%%%%%%%%%%%%%%%%%%%%%%%%%%%%%%%%%%%%%%%%%%%%%%%%%%%%%%%%%%%%%%%%%%%%%%%%
%%%%%%%%%%%%%%%%%%%%%%%%%%%%%%%%%%%%%%%%%%%%%%%%%%%%%%%%%%%%%%%%%%%%%%%%%%%%%%

\section{Modules over categories}%, free resolutions, induction~and~restriction}
\label{MOD-CAT}

In this and the following two sections, we review the definitions of modules over a category, free modules, tensor products, and induction and restriction functors.
Our presentation is self-contained, but for additional details and examples the reader may consult~\cite{bettiposets} or \citelist{\cite{tomDieck}*{Section~I.11} \cite{Lueck}*{Section~9}}.
We also describe the equivalence between graded modules and modules over the associated action category in \autoref{mod-over-act-cat}, which is crucial to our approach, and the functorial bar resolution in \autoref{bar}.

Fix a commutative, associative, and unital ring~$\gf$, and denote by~$\modules{\gf}$ the corresponding abelian category of modules over~$\gf$ and $\gf$-linear homomorphisms.
Let $\cat$ be a small category.
The category~$\modules{\gf\cat}$ of \emph{modules over~$\cat$} is defined as the functor category $\Fun(\cat,\modules{\gf})$.
Explicitly: a $\gf\cat$-module is a functor from $\cat$ to~$\modules{\gf}$; a homomorphism of $\gf\cat$-modules is a natural transformation of such functors.
Notice that, if $M$ is a $\gf\cat$-module and $c$ and~$d$ are objects of~$\cat$, then there is a $\gf$-linear ``evaluation'' homomorphism
\begin{equation}
\label{eq:adj}
\gf[\mor_\cat(c,d)]\tensor M(c)
\TO
M(d)
\,,
\end{equation}
which we call a \emph{structure map} of~$M$, whose adjoint is obtained by linear extension of the function~$\mor_\cat(c,d)\TO\hom_\gf(M(c),M(d))$ expressing the functoriality of~$M$.

The category~$\modules{\gf\cat}$ is an abelian category, with kernels and images computed object-wise.
A sequence of $\gf\cat$-modules~$L\TO M\to N$ is exact if and only if $L(c)\TO M(c)\TO N(c)$ is exact for each $c\in\obj\cat$.

\begin{example}[the constant module]
The \emph{constant} $\gf\cat$-module is the functor
\[
\MOR{\const_\gf}{\cat}{\modules{\gf}}
\]
given by $\const_\gf(c)=\gf$ and $\const_\gf(u)=\id_\gf$ for all objects $c$ and all morphisms $u$ of the category~$\cat$.
\end{example}

In order to study graded modules, we are led to consider the following categories.

\begin{definition}[action categories]
\label{act-cat-def}
Given a monoid~$\Lambda$ and a set~$S$ equipped with a left action of~$\Lambda$, the \emph{action category}~$\act{\Lambda}{S}$ is the category with
\[
\obj\act{\Lambda}{S}=S
\AND
\mor_{\act{\Lambda}{S}}(s,t)=\SET{\lambda\in\Lambda}{\lambda s=t}
\,,
\]
and with composition given by multiplication in~$\Lambda$.
\end{definition}

\begin{example}
\label{act-cat-examples}
Given any monoid $\Lambda$ we can consider the action category~$\act{\Lambda}{\Lambda}$ of~$\Lambda$ acting on itself.
More generally, if $\MOR{f}{\Lambda}{\Gamma}$ is a homomorphism of monoids, then $\Lambda$ acts on~$\Gamma$ via~$f$, and we can form the action category~$\act{\Lambda}{\Gamma}$.

Now assume that $\Gamma$ is a group. %or.
Then $\mor_{\act{\Lambda}{\Gamma}}(\gamma,\gamma')=f^{-1}(\{\gamma'\gamma^{-1}\})$.
Moreover, if we also assume that $f$ is injective, then $\act{\Lambda}{\Gamma}$ is just the preorder~$(\Gamma,\leq)$, considered as a category, where $\gamma\le\gamma'$ if and only if $\gamma'=f(\lambda)\gamma$ for some (unique) $\lambda\in\Lambda$.
This relation is a partial order (i.e., $(\Gamma,\leq)$ is a poset) if and only if there are no nontrivial invertible elements in~$\Lambda$.
\end{example}

The next result follows at once from the definitions.

\begin{lemma}[graded modules as functors over action categories]
\label{mod-over-act-cat}
Let $\Lambda$ be a monoid, and let $\MOR{\deg}{\Lambda}{\Gamma}$ be a homomorphism of monoids.
Consider the monoid ring~$\gf[\Lambda]$ with the induced $\Gamma$-grading.
Then the functor
\[
\modules{
\gf\bigl(%[
\act{\Lambda}{\Gamma}\bigr)%]
}
\TO
\grmodules{\Gamma}{\gf[\Lambda]}
\,,\qquad
M\longmapsto\bigoplus_{\gamma\in\Gamma}M(\gamma)
\]
from the category of modules over the action category~$\act{\Lambda}{\Gamma}$ to the category of $\Gamma$\nobreakdash-graded $\gf[\Lambda]$-modules and $\Gamma$-graded homomorphisms is an equivalence of abelian categories.
\end{lemma}

\begin{example}[multigraded modules over polynomial rings]
\label{mod-over-act-cat-2}
Consider the monoid $\Lambda=\IN^m$ and the inclusion homomorphism $\MOR{\deg=\incl}{\IN^m}{\IZ^m}$, where $m\ge1$ is an integer.
Notice that in this case the action category~$\act{\IN^m}{\IZ^m}$ is just the poset~$\IZ^m$, viewed as a category, with respect to the usual component-wise partial order: $(a_1,\dotsc,a_m)\leq(b_1,\dotsc,b_m)$ if and only if $a_i\leq b_i$ for each~$i$;
compare \autoref{act-cat-examples}.
Moreover, identifying the elements $(a_1,\dots, a_m)$ of $\IN^m$ with the
monomials $x_1^{a_1}\dots x_m^{a_m}$, we see that
$\IZ^m$-graded $\gf[\IN^m]$-modules are nothing but multigraded modules over the polynomial ring~$\gf[x_1,\dotsc,x_m]$, and \autoref{mod-over-act-cat} is the well-known observation that these are equivalent to functors $\IZ^m\TO\modules{\gf}$.
\end{example}

\begin{example}[graded modules over polynomial rings]
\label{graded-modules-example}
Consider %the monoid
$\Lambda=\IN^m$, where $m\ge1$ is an integer,
and the standard degree homomorphism $\MOR{\deg}{\IN^m}{\IZ}$ given by
$(a_1,\dots,a_m)\mapsto a_1+\dots+a_m$.
In this case the action category~$\act{\IN^m}{\IZ}$ is not a poset:
its objects are the integers, the morphisms from $k$ to $l$ can be
identified with the set of all monomials of degree $l-k$ in the polynomial
ring $\gf[x_1,\dots,x_m]$, and composition of morphisms is just multiplication
of monomials.
%viewed as a category, with respect to the usual component-wise partial
%order: $(a_1,\dotsc,a_m)\leq(b_1,\dotsc,b_m)$ if and only if $a_i\leq b_i$
%for each~$i$;
%compare \autoref{act-cat-examples}.
%Since $\IZ$-graded $\gf[\IN^m]$-modules are nothing but graded
%modules over the polynomial ring~$\gf[x_1,\dotsc,x_m]$, %and
%\autoref{mod-over-act-cat} is the well-known observation that these are %equivalent to functors $\IZ\TO\modules{\gf}$.
\end{example}

%%%%%%%%%%%%%%%%%%%%%%%%%%%%%%%%%%%%%%%%%%%%%%%%%%%%%%%%%%%%%%%%%%%%%%%%%%%%%%

\section{Free modules}
\label{FREE}

Here we recall the notion of free $\gf\cat$-modules.
Any $\gf\cat$-module~$M$ has an underlying collection of sets~$M(c)$ indexed by the objects $c\in\obj\cat$, and this data can be thought of as a functor to the category of sets from the discrete category~$\obj\cat$ (i.e., the subcategory of~$\cat$ whose only morphisms are the identities).
We call such functors $(\obj\cat)$-sets, and we obtain the \emph{forgetful functor}
\[
\MOR{\forget}{\modules{\gf\cat}}{(\obj\cat)\D\Sets}
\,.
\]
This forgetful functor~$\forget$ has a left adjoint
\[
\MOR{\free}{(\obj\cat)\D\Sets}{\modules{\gf\cat}}
\,,
\]
which is defined by sending an $(\obj\cat)$-set~$B$ to the $\gf\cat$-module
\[
\free B=
\adjustlimits
\bigoplus_{c\in\obj\cat}
\bigoplus_{\ B(c)\ }
\gf[\mor_\cat(c,-)]
\,.
\]
The unit of the adjunction is the natural transformation % of $(\obj\cat)$-sets
\(
\MOR{\eta}{B}{\forget\free B}
\)
that for each $c\in\obj\cat$ sends $b\in B(c)$ to~$\id_c$ in the direct summand~$\gf[\mor_\cat(c,c)]\leq\free B(c)$ indexed by~$b$.

\begin{definition}
We say that the $\gf\cat$-module~$\free B$ is \emph{free with basis} $\MOR{\eta}{B}{\forget\free B}$, and we define a $\gf\cat$-module to be \emph{free} if it is isomorphic to one in the image of the functor~$\free$.
\end{definition}

It is a standard exercise to verify that free $\gf\cat$-modules are projective, and that a $\gf\cat$-module is projective if and only if it is a direct summand of a free $\gf\cat$-module.

\begin{example}[free module on one generator]
(a)
Let $c$ be an object of $\cat$, and let $B$ be the
$\obj\cat$-set with $B(d)=\emptyset$ for $d\ne c$ and with $B(c)=pt$.
Let $\MOR{\eta}{B}{\gf[\mor_\cat(c,-)]}$ be given by $\eta(pt)=\id_c$.
This makes the $\gf\cat$-module $\gf[\mor_\cat(c,-)]$ free, and we call it
a \emph{free module on one generator} (based at $c$).

(b)
Under  the equivalence of Lemma~\ref{mod-over-act-cat}, a
free $\gf(\act{\Lambda}{\Gamma})$-module on one generator based at $c\in\Gamma$
corresponds to a free $\Gamma$-graded $\gf[\Lambda]$-module on one homogeneous generator of degree $c$.
\end{example}

\begin{definition}
Given a small category~$\cat$, we denote by~$\cla$ the set of isomorphism classes of objects of~$\cat$, and by~$\MOR{\cl}{\obj\cat}{\cla}$ the function that sends each object to its isomorphism class.
We say that an $(\obj\cat)$-set $B$ is \emph{of finite type} if
\[
\#\SET{\cl(c)\in\cla}{B(c)\neq\emptyset}<\infty
\,;
\]
we say that $B$ is \emph{finite} if it is of finite type and moreover $\#B(c)<\infty$ for each $c\in\obj\cat$.
We say that a $\gf\cat$-module~$M$ is \emph{finitely generated} (or \emph{of finite type}) if $M$ is a quotient of a free module with a finite (or of finite type, respectively) basis.
\end{definition}

\begin{definition}[unnormalized bar resolution]
\label{bar}
Let $M$ be a $\gf\cat$-module.

(a) For each $n\ge0$, define a $\gf\cat$-module~$F_n^u$ by setting
\[
F_n^u(-)=
\bigoplus_{c_0\to\dotsb\to c_n}
\gf[\mor_\cat(c_n,-)]\tensor M(c_0)
\,,
\]
where the direct sum is indexed over all $n$-tuples of composable morphisms of~$\cat$
\[
c_0\TO[u_1] c_1\TO\dotsb\TO c_{n-1}\TO[u_n] c_n
\,,
\]
i.e., all $n$-simplices of the nerve of~$\cat$.
When $n=0$ this reduces to
\[
F_0^u(-)=
\bigoplus_{c_0\in\obj\cat}
\gf[\mor_\cat(c_0,-)]\tensor M(c_0)
\,,
\]
and the structure maps~\eqref{eq:adj} of~$M$ yield a $\gf\cat$-module homomorphism $\MOR{\varepsilon}{F_0^u}{M}$.

(b) For each $n>0$ and $0\le i\le n$, define morphisms
\[
\MOR{\partial_n^i}{F_n^u}{F_{n-1}^u}
\]
as follows:
\begin{itemize}
\item
If $0<i<n$, then $\partial_n^i$ sends the summand indexed by
\[
c_0\TO\dotsb\TO c_{i-1}\TO[u_{i\vphantom{+1}}]c_i\TO[u_{i+1}] c_{i+1}\TO\dotsb\TO c_n
\]
to the summand indexed by
\[
c_0\TO\dotsb\TO c_{i-1}\xrightarrow{u_{i+1}u_i} c_{i+1}\TO\dotsb\TO c_n
\]
using the identity on~$\gf[\mor_\cat(c_n,-)]\tensor M(c_0)$.
Here and in the following, the unlabeled morphisms are left unchanged.

\item
If $i=0$, then $\partial_n^0$ sends the summand indexed by
\(
c_0\TO[\smash{u_1}]c_1\TO\dotsb\TO c_n
\)
to the~summand indexed by
\(
c_1\TO\dotsb\TO c_n
\)
using the map
\[
\MOR{\id\tensor M(u_1)}%
{\gf[\mor_\cat(c_n,-)]\tensor M(c_0)}%
{\gf[\mor_\cat(c_n,-)]\tensor M(c_1)}
\,.
\]

\item
If $i=n$, then $\partial_n^n$ sends the summand indexed by
\(
c_0\TO\dotsb\TO c_{n-1}\TO[\smash{u_n}]c_n
\)
to the summand indexed by
\(
c_0\TO\dotsb\TO c_{n-1}
\)
using the map
\[
\MOR{u_n^*\tensor\id}%
{\gf[\mor_\cat(c_n,    -)]\tensor M(c_0)}%
{\gf[\mor_\cat(c_{n-1},-)]\tensor M(c_0)}
\,.
\]
\end{itemize}
We then let
\[
\MOR{d_n=\sum_{i=0}^n(-1)^i\partial_n^i}{F_n^u}{F_{n-1}^u}
\,,
\]
thus obtaining a chain complex~$\cx{B}^u(M)$ of $\gf\cat$-modules, together with an augmentation~$\MOR{\varepsilon}{F_0^u}{M}$.
% $\varepsilon\partial_1^0=\varepsilon\partial_1^1$
\end{definition}

\begin{remark}\label{degeneracy-remarks}
(a)
The verification that $\cx{B}^u(M)$ is a chain complex is straightforward.
In fact, $\cx{B}^u(M)$ is the associated (unnormalized) chain complex of the semi-simplicial $\gf\cat$-module defined by the maps $\partial_n^i$ above; compare for example \cite{Weibel}*{8.1.9, 8.1.6, and~8.2.1 on pages 258--260}.

(b)
There are also degeneracy maps~$\MOR{\sigma_n^i}{F_n}{F_{n+1}}$, defined by inserting identities in the obvious ways, making~$\cx{B}^u(M)$ into a simplicial $\gf\cat$-module.
\end{remark}

\begin{definition}[normalized bar resolution]
\label{normalized-bar}
Let $M$ be a $\gf\cat$-module. The \emph{degenerate} chain
complex $\cx{D}(M)$ is the subcomplex of $\cx{B}^u(M)$ generated by
the degeneracy maps $\sigma_n^i$ from
Remark~\ref{degeneracy-remarks}(b). Thus
in homological degree $n$ we have
\[
D_n(M)=\sum_i \sigma_{n-1}^i(F_{n-1}^u).
\]
We define the \emph{(normalized) bar complex}
$\cx{B}(M)$ as the quotient chain complex
\[
\cx{B}(M) =
\cx{B}^u(M)/\cx{D}(M).
\]
In particular, the component $F_n$ of $\cx{B}(M)$
in homological degree $n$ is the $\gf\cat$-module
\[
F_n(-)=
\bigoplus_{c_0\to\dotsb\to c_n}
\gf[\mor_\cat(c_n,-)]\tensor M(c_0)
\,,
\]
where the direct sum is indexed over all $n$-tuples of composable
non-identity morphisms of~$\cat$,
%\[
%c_0\TO[u_1] c_1\TO\dotsb\TO c_{n-1}\TO[u_n] c_n
%\,,
%\]
%with no $u_i$ the identity,
i.e., all non-degenerate $n$-simplices of the nerve of~$\cat$.
\end{definition}

\begin{proposition}
%We claim that
The complexes $\cx{B}^u(M)$ and $\cx{B}(M)$ are resolutions of~$M$.
\end{proposition}

\begin{proof}
By \cite{Weibel}*{8.3.6, 8.3.8 on pages 265--266} it is enough to
show that $\cx{B}^u(M)$ is a resolution of $M$.
In order to see this, notice that we can define extra maps
\[
\ANDalign{\MOR{\gamma_{-1}}{M(c)}{F_0(c)}}{\MOR{\gamma_n}{F_n(c)}{F_{n+1}(c)}}
\]
for each~$c\in\obj\cat$ and $n\ge0$ as follows.
The map~$\gamma_{-1}$ is induced by the map
\[
M(c)\TO\gf[\mor_\cat(c,c)]\tensor M(c)
\,,\qquad
m\longmapsto \id_c\tensor m
\,.
\]
To define~$\gamma_n$, given any
\(
c_0\TO\dotsb\TO c_n
\)
we need to construct a map
\[
\gf[\mor_\cat(c_n,c)]\tensor M(c_0)
\TO
F_{n+1}(c)
=
\bigoplus_{d_0\to\dotsb\to d_{n+1}}
\gf[\mor_\cat(d_{n+1},c)]\tensor M(d_0)
\,,
\]
or, equivalently, a function
\[
\mor_\cat(c_n,c)
\TO
\hom_\gf\!\!\left(
M(c_0),
\bigoplus_{\;d_0\to\dotsb\to d_{n+1}}
\gf[\mor_\cat(d_{n+1},c)]\tensor M(d_0)
\right)
\!.
\]
To define this latter function, we send $\MOR{u}{c_n}{c}$ to the homomorphism
\[
M(c_0)\TO\gf[\mor_\cat(c,c)]\tensor M(c_0)
\,,\qquad
m\longmapsto \id_c\tensor m
\]
in the summand indexed by
\(
c_0\TO\dotsb\TO c_n\TO[\smash{u}]{c}\,.
\)
It is easy to check that $\varepsilon\gamma_{-1}=\id$, $\partial_1^0\gamma_0=\gamma_{-1}\varepsilon$, and, for all $n>0$ and all~$0\le i\le n$, $\partial_n^n\gamma_{n-1}=\id$ and $\partial_{n+1}^i\gamma_n=\gamma_{n-1}\partial_n^i$.

Therefore the augmented simplicial $\gf$-module $\cx{B}(M)(c)\TO M(c)$ is (right) contractible in the sense of~\cite{Weibel}*{page~275}, and so the associated augmented chain complex is a resolution~\cite{Weibel}*{8.4.6.1 on page~275}.
\end{proof}

\begin{remark}
%We remark that
(a) The maps~$\gamma_n$ are not natural in~$c$, i.e., they do not define homomorphisms of~$\gf\cat$-modules.

(b) Observe that we can view
\[
F_n(-)=
\adjustlimits
\bigoplus_{c_n\in\obj\cat}
\bigoplus_{\;c_0\to\dotsb\to c_n}
\gf[\mor_\cat(c_n,-)]\tensor M(c_0)
\,,
\]
where the first direct sum is indexed over all objects~$c_n$ of~$\cat$, and the second one is indexed over all $n$-tuples of composable non-identity morphisms of~$\cat$ that end in~$c_n$.
This shows that, if $M$ is object-wise free, i.e., if the $\gf$-module~$M(c)$ is free over~$\gf$ for each $c\in\obj\cat$, then each~$F_n$ is a free $\gf\cat$-module, and so $\cx{B}(M)$ is a free resolution of~$M$.
Obviously, this condition on~$M$ is automatically satisfied when $\gf$ is a field.
\end{remark}

%%%%%%%%%%%%%%%%%%%%%%%%%%%%%%%%%%%%%%%%%%%%%%%%%%%%%%%%%%%%%%%%%%%%%%%%%%%%%%

\section{Tensor products, induction and restriction}
\label{IND-RES}

Now we recall the definition of tensor products of modules over categories, and then use them to study induction and restriction along functors.

Given a $\gf\cat^\op$-module~$N$ and a $\gf\cat$-module~$M$, their \emph{tensor product over~$\gf\cat$} is the $\gf$-module $N\tensor_{\gf\cat}M$ defined as the quotient
\[
\bigoplus_{c\in\obj\cat}N(c)\tensor M(c)
\Bigm/
\bigl\langle\,
nu\tensor m - n\tensor um \bigm| n\in N(d), m\in M(c), u\in\mor_\cat(c,d)
\,\bigr\rangle
\,,
\]
where, to underscore the analogy with the tensor product of modules over rings, we use the shorthands $um$ and $nu$ for the values of the homomorphisms $M(u)$ and $N(u)$ at the elements $m$ and $n$ respectively; i.e., $um=M(u)(m)$ and $nu=N(u)(n)$.
More conceptually, $N\tensor_{\gf\cat}M$ is the coend of the functor
\[
\cat^\op\times\cat\TO\modules{\gf}
\,,\qquad
(d,c)\longmapsto N(d)\tensor M(c)
\,.
\]

Notice that if $\catwo$ is another small category, $N$ is a $\gf(\cat^\op\times\catwo)$-module, and $M$ is a $\gf\cat$-module, then $N\tensor_{\gf\cat}M$ becomes a $\gf\catwo$-module.
We think of and refer to $\gf(\cat^\op\times\catwo)$-modules as \emph{$\gf\catwo$-$\gf\cat$-bimodules}.
Dually, given a $\gf\catwo$-module~$L$, we get a $\gf\cat$-module $\hom_{\gf\catwo}(N,L)$ of $\gf\catwo$-homomorphisms; notice that $\hom_{\gf\catwo}(N,L)$ is covariant in~$\cat$ because $N$ is contravariant in~$\cat$.
It is easy to see that the functors
\[
\MOR{N\tensor_{\gf\cat}-}{\modules{\gf\cat}}{\modules{\gf\catwo}}
\AND
\MOR{\hom_{\gf\catwo}(N,-)}{\modules{\gf\catwo}}{\modules{\gf\cat}}
\]
are adjoint, i.e., for all $\gf\cat$-modules~$M$ and all $\gf\catwo$-modules~$L$, there are natural isomorphisms
\[
\hom_{\gf\catwo}\bigl(N\tensor_{\gf\cat}M,L\bigr)
\cong
\hom_{\gf\cat}\bigl(M,\hom_{\gf\catwo}(N,L)\bigr)
\,.
\]
It follows that that $N\tensor_{\gf\cat}-$ is right exact and $\hom_{\gf\catwo}(N,-)$ is left exact.

% \begin{remark}[two-sided bar construction]
% Generalizing the bar resolution of \autoref{bar},
% given a $\gf\cat^\op$-module~$N$ and a $\gf\cat$-module~$M$, define for each $n\ge0$
% \[
% B_n(N,\cat,M)=
% \bigoplus_{c_0\to\dotsb\to c_n}
% N(c_n)\tensor M(c_0)
% \,,
% \]
% and maps $\partial_n^i$ and $\MOR{d_n}{B_n(N,\cat,M)}{B_{n-1}(N,\cat,M)}$ using the same formulas as in \autoref{bar}.
% With this notation, for each $c\in\obj\cat$ the chain
% complex~$\cx{B}(M)(c)$ from \autoref{bar} is equal to $\cx{B}(\gf[\mor_\cat(-,c)],\cat,M)$.

% Notice that $N\tensor_{\gf\cat}M$ is by definition the cokernel of~$d_1$, i.e., there is a natural isomorphism
% $
% %\[
% H_0\bigl(\cx{B}(N,\cat,M)\bigr)\cong N\tensor_{\gf\cat}M
% \,.
% %\]
% $
% \end{remark}

\begin{definition}
Given a functor  $\MOR{\alpha}{\cat}{\catwo}$, we
define the \emph{tautological} $\gf\catwo$-$\gf\cat$-bimodule~$\gf\catwo_\alpha$ as follows:
\[
\MOR{\gf\catwo_\alpha}{\cat^\op\times\catwo}{\modules{\gf}}
\,,\qquad
(c,d)\longmapsto\gf\bigl[\mor_\catwo(\alpha(c),d)\bigr]
\,.
\]
When $\alpha$ is the identity functor $\MOR{\id_\cat}{\cat}{\cat}$ we omit
the subscript and write simply $\gf\cat$ for the tautological bimodule $\gf\cat_{\id_\cat}$.
\end{definition}

Precomposition with $\MOR{\alpha}{\cat}{\catwo}$ defines a functor
\[
\MOR{\res_\alpha}{\modules{\gf\catwo}}{\modules{\gf\cat}}
\,,\qquad
L\longmapsto L\circ\alpha
\,,
\]
which we call \emph{restriction along~$\alpha$}, and which is obviously exact.
Notice that, for each $\gf\catwo$-module~$L$ and each $c\in\obj\cat$, we have
\[
\hom_{\gf\catwo}\bigl(\gf\catwo_\alpha,L\bigr)(c)
=
\hom_{\gf\catwo}\bigl(\gf\bigl[\mor_\catwo(\alpha(c),-)\bigr],L(-)\bigr)
\cong
L(\alpha(c))
\,,
\]
by Yoneda's Lemma, and so
\[
\res_\alpha \cong \hom_{\gf\catwo}\bigl(\gf\catwo_\alpha,-\bigr)
\,.
\]
Therefore the functor
\[
\MOR{\ind_\alpha=\gf\catwo_\alpha\tensor_{\gf\cat}-}{\modules{\gf\cat}}{\modules{\gf\catwo}}
\,,
\]
which we call \emph{induction along~$\alpha$}, is left adjoint to~$\res_\alpha$, and hence $\ind_\alpha$ is right exact.
More conceptually, $\ind_\alpha$ is the left Kan extension along~$\alpha$.
Being left adjoint to an exact functor, induction preserves projective modules.
Induction also preserves free modules: for any fixed $c_0\in\obj\cat$, and for each~$d\in\obj\catwo$, we have
\[
\ind_\alpha\bigl(\gf[\mor_\cat(c_0,-)]\bigr)(d)
=
\gf[\mor_\catwo(\alpha(-),d)]\!\tensor_{\gf\cat}\!\gf[\mor_\cat(c_0,-)]
\cong
\gf[\mor_\catwo(\alpha(c_0),d)]
\,,
\]
where the last isomorphism is a consequence of Yoneda's Lemma.
Since induction is right exact, induction also preserves finitely generated modules.
On the other hand, restriction does not preserve free nor projective nor finitely generated modules in general.

% \begin{example}
% Let $\cat'$ be a full subcategory of $\cat$, and let
% $\MOR{\varrho}{\cat'}{\cat}$ be the inclusion functor. Let
% $F(-)=\gf[\mor_{\cat'}(c',-)]$ be a free $\gf\cat'$-module on one
% generator. Then $\ind_\varrho F\cong \gf[\mor_\cat(c',-)]$ is the
% free $\gf\cat$-module on the same generator (now considered as
% an $\obj\cat$-set).
% \end{example}

\begin{remark}
An equivalence of categories $\MOR{\alpha}{\CC}{\CD}$ induces an equivalence of module categories $\MOR{\res_\alpha}{\modules{\gf\catwo}}{\modules{\gf\cat}}$.
More explicitly, if $\MOR{\beta}{\catwo}{\cat}$ is a functor such that the compositions $\beta\alpha$ and $\alpha\beta$ are naturally isomorphic to the respective identities, then the same is true for $\res_\beta\res_\alpha=\res_{\beta\alpha}$ and $\res_\alpha\res_\beta=\res_{\alpha\beta}$.
Moreover, the uniqueness of adjoints implies that there are natural isomorphisms $\ind_\alpha\cong\res_\beta$ and $\ind_\beta\cong\res_\alpha$.
\end{remark}

%%%%%%%%%%%%%%%%%%%%%%%%%%%%%%%%%%%%%%%%%%%%%%%%%%%%%%%%%%%%%%%%%%%%%%%%%%%%%%

\section{EI~categories, splitting functor, and supports}
\label{EI}

An \emph{EI~category} is a category in which all \emph{E}ndomorphisms are \emph{I}somorphisms.
Notice that, in any EI~category, if there exists an isomorphism $c\TO d$, then every morphism $c\TO d$ is an isomorphism.

Given a small category~$\cat$, we denote by~$\cla$ the set of isomorphism classes of objects of~$\cat$, and by~$\MOR{\cl}{\obj\cat}{\cla}$ the function that sends each object to its isomorphism class.
Define a preorder on~$\cla$ be setting $\cl(c)\le\cl(d)$ if and only if there exists a morphism $c\TO d$.
If $\cat$ is an EI~category, then the relation~$\leq$ is also antisymmetric, and so $(\,\cla,\,\leq\,)$ is a poset; in that case
$\cl(c)<\cl(d)$ if and only if there exists a non-isomorphism $c\TO d$.
When $\cat$ is EI we write $\cat_{\le c}$ for the full subcategory of $\cat$ with objects all $d$ such that $\cl(d)\le\cl(c)$; the category $\cat_{<c}$
is defined analogously.

We recall now some standard notions from order theory.
Let $(\CP,\leq)$ be a preorder, %poset,
e.g., the set of isomorphism classes of objects in a small %EI~
category.
A subset $A\subseteq\CP$ is called \emph{Artinian} if every descending chain in~$A$ stabilizes.
Equivalently, $A$ is Artinian if and only if every non-empty subset of~$A$ has a minimal element.
Obviously, the property of being Artinian passes to subsets, and it is
straightforward to see that an Artinian subset is a poset.
A subset $A\subseteq\CP$ is called an \emph{upper set} if $a\in A$ and $a\leq b$ imply $b\in A$.
The smallest upper set containing a subset~$S\subseteq\CP$ is denoted~$\upset{S}$ and called the \emph{upper set generated by~$S$}.
We say that an upper set~$A\subseteq\CP$ is \emph{finitely generated} if there exists a finite set~$S\subseteq A$ such that $\upset{S}=A$.

We denote by~$\iso$ the subcategory of~$\cat$ with the same objects but with only the isomorphisms of~$\cat$ as morphisms.
In other words, $\iso$ is the maximal groupoid inside~$\cat$. We write
$\iota$ for the inclusion functor
\(
\MOR{\iota}{\iso}{\cat}.
\)
Given an object $c\in\obj\cat$, let $\aut_\cat(c)$ denote the group of automorphisms of~$c$ in~$\cat$, and let $\gf[c]=\gf[\aut_\cat(c)]$ be the corresponding group algebra.
If $\cat$ is an EI~category, then
\(
\aut_\cat(c)=\mor_{\cat}(c,c)=\mor_{\iso}(c,c)
\).
We remark that, by choosing a representative for each isomorphism class of objects in~$\cat$, we obtain an equivalence of categories
\[
\modules{\gf\iso}
\simeq
\prod_{\cl(c)\in\cla}\modules{\gf[c]}
\,.
\]

\begin{definition}
We denote by
\[
\MOR{\ind}{\modules{\gf\iso}}{\modules{\gf\cat}}
\AND
\MOR{\res}{\modules{\gf\cat}}{\modules{\gf\iso}}
\]
the adjoint induction and restriction functors along the inclusion $\MOR{\iota}{\iso}{\cat}$.
\end{definition}

\begin{remark}\label{ind-on-isoC}
(a)
Let $Q$ be a $\gf\iso$-module. For each $c,e\in\obj\cat$ the group
$\aut_\cat(c)$ acts on $\gf[\mor_\cat(c,e)]\otimes Q(c)$, where
$u\in\aut_\cat(c)$ sends $v\otimes m$ to $vu^{-1}\otimes Q(u)m$. %Now
A~routine computation using the
definitions shows that %the $\gf\cat$-module $\ind Q$ satisfies
\[
\ind Q (e) \cong
\bigoplus_{\cl(c)\in\cla}
\Bigl(\gf[\mor_\cat(c,e)]\otimes Q(c)\Bigr)\big/\aut_\cat(c),
\]
and for $w\in\mor_\cat(e,f)$ the morphism
$\MOR{\ind Q(w)}{\ind Q(e)}{\ind Q(f)}$
is induced by composition with $w$. Furthermore, the
unit of adjunction
\[
\MOR{\eta_Q(e)}{Q(e)}{\res\ind Q(e)}
\]
corresponds
to the inclusion
$
Q(e)\TO \id_e\otimes Q(e)\subseteq
\Bigl(\gf[\mor_\cat(e,e)]\otimes Q(e)\Bigr)\big/\aut_\cat(e).
$

(b)
Let $M$ be a $\gf\cat$-module. The counit of the adjunction
\[
\MOR{\varepsilon_M(e)}{\ind\res M(e)}{M(e)}
\]
is the morphism
\[
\bigoplus_{\cl(c)\in\cla}
\Bigl(\gf[\mor_\cat(c,e)]\otimes M(c)\Bigr)\big/\aut_\cat(c) \quad
\TO
\quad M(e)
\]
induced by  the structure maps \eqref{eq:adj} of $M$.
\end{remark}

\begin{definition}[splitting functor]
Given a $\gf\cat$-module~$M$, for each $c\in\obj\cat$ let
\[
\jrad M(c)
=
\image
\left(\
%{\bigl(
%\Sigma\ M(u)
%\bigr)%_u
%}
%\colon
{\bigoplus_{\substack{b\TO[u]c\\\text{non-iso}}}M(b)}
\xrightarrow{\quad\sum M(u)\quad}
{M(c)}
\
\right)
,
\]
and define $\spl M(c)= M(c)/\jrad M(c)$.
Given any isomorphism $c\TO[v]d$, the composition $b\TO[u]c\TO[v]d$ is an isomorphism if and only if $u$ is an isomorphism.
Therefore we get two functors $\MOR{\jrad M, \spl M}{\iso}{\modules{\gf}}$, hence also
%In particular, this yields
a functor
\[
\MOR{\jrad}{\modules{\gf\cat}}{\modules{\gf\iso}}
\,,
\]
which is a category-theoretic analogue of the
ring-theoretic notion of Jacobson radical,  and
a functor
\[
\MOR{\spl}{\modules{\gf\cat}}{\modules{\gf\iso}}
\,,
\]
which is called the \emph{splitting functor}.
\end{definition}

\begin{remark}\label{ind-spl}
%Notice that,
(a)
It is straightforward from the definition that both
$\jrad$ and $\spl$ preserve epimorphisms and direct sums.

(b)
By definition, for each $\gf\cat$-module~$M$
there is a natural short exact sequence
\[
0 \TO \jrad M \TO[i_M] \res M \TO[p_M] \spl M\TO 0
\]
of $\gf\iso$-modules.  %for each $\gf\cat$-module~$M$.
Since $\res$ is exact and $\jrad$ and $\spl$ preserve
epimorphisms, a routine diagram chase shows that
$\spl$ is a right exact functor.

(c) Let $Q$ be a $\gf\iso$-module. A direct consequence from the
observations in \autoref{ind-on-isoC} and the definitions is  that
\[
\spl\ind Q(e)=
\begin{cases}
0    &\text{ if $\id_e=uw$ for some non-isomorphisms $u,w\in\mor_\cat(e,e)$}; \\
Q(e) &\text{ otherwise. }
\end{cases}
\]
In particular $\spl\ind Q$ is a quotient of $Q$, and it is immediate
that the composition
\[
Q \TO[\eta_Q]  \res\ind Q \xrightarrow{p_{\ind Q}} \spl\ind Q
\]
is precisely the natural quotient epimorphism $\MOR{\tau_Q}{Q}{\spl\ind Q}$.

(d) Let $M$ be a $\gf\cat$-module. Then $\spl M(e)=0$ whenever
$\id_e=uw$ for some non-isomorphisms $u,w\in\mor_\cat(e,e)$.
Furthermore, we have
\[
\spl(\varepsilon_M)\circ\tau_{\res M} = p_M.
\]
Indeed, due to the naturality of the morphism $p$ and the basic properties
of the unit and counit of adjunction, we have
$
\spl(\varepsilon_M)\circ\tau_{\res M} =
\spl(\varepsilon_M)\circ p_{\ind\res M}\circ \eta_{\res M} =
p_M\circ \res(\varepsilon_M)\circ \eta_{\res M} =
p_M.
$
\end{remark}

\begin{remark}
Let $\cat$ be a small EI category.

(a)
We have an equivalent description of the splitting functor as follows.
Define a $\gf\iso$-$\gf\cat$-bimodule~$B$ by
\begin{align*}
\MOR{B}{\cat^\op\times\iso&}{\modules{\gf}}
\,,
\\
(c,d)&\longmapsto \gf\bigl[\mor_{\iso}(c,d)\bigr].
\end{align*}
Since
\[
\gf\bigl[\mor_{\iso}(c,d)\bigr] =
\begin{cases}
\gf\bigl[\mor_{\cat}(c,d)\bigr] &\text{if $c\cong d$\,,}
\\
0&\text{if $c\not\cong d$\,,}
\end{cases}
\]
we obtain a natural isomorphism
\[
%\(
\spl M\cong B\tensor_{\gf\cat}M
\,.
%\)
\]
From this description we see that when $\cat$ is EI the functor
%the functor $\spl$ is right exact. Moreover,
$\spl$ sends free, projective, and finitely generated $\gf\cat$-modules to $\gf\iso$-modules with the same properties.

(b)
%\begin{remark}
%Finally, notice that,
For each $\gf\iso$-module~$Q$,
%(a)
%Suppose $\cat$ is an EI category. Then
the natural epimorphism %of functors
$\MOR{\tau_Q}{Q}{\spl\ind Q}$ from Remark~\ref{ind-spl}(c)
%where for a $\gf\iso$-module~$Q$
%the map $\MOR{\tau_Q}{Q}{\spl\ind Q}$
is given by
the composition
\begin{equation}
\label{eq:spl-ind}
Q
\cong
\gf\iso\tensor_{\gf\iso}Q
\cong
%B\tensor_{\gf\cat}\Bigl(\gf\cat_\alpha\tensor_{\gf\iso}Q\Bigr)
%\cong
\Bigl(B\tensor_{\gf\cat}\gf\cat_\iota\Bigr)\tensor_{\gf\iso}Q
\cong
B\tensor_{\gf\cat}\Bigl(\gf\cat_\iota\tensor_{\gf\iso}Q\Bigr)
\cong
%\gf\iso\tensor_{\gf\iso}Q
%\cong
\spl\ind Q
\,,
\end{equation}
%It is now straightforward from the definitions to check
%for any $\gf\cat$-module $M$
%that the counit of the adjunction $\MOR{\varepsilon_M}{\ind\res M}{M}$
%and any $\gf\cat$-module $M$
%satisfies  $(\spl\varepsilon_M)(\tau_{\res M})=p_M$.
in particular it is an isomorphism.
%Indeed since both $\ind$ and $\spl$ preserve
%free modules and direct sums, this is true for
\end{remark}

We conclude this section by introducing several notions of support that
we will need later.

\begin{definition}[support]
The \emph{support} of an $\obj\cat$-set $B$ is the following subset
of $\cla$:
\[
\supp(B)=\SET{\cl(c)\in\cla}{B(c)\neq\emptyset}.
\]
The \emph{support} and the \emph{minimal support} of a $\gf\cat$-module~$M$ are the following subsets of~$\cla$:
\begin{align*}
\supp(M)
&=
\SET{\cl(c)\in\cla}{M(c)\neq0}
\,,
\\
\minsupp(M)
&=
\SET{\cl(c)\in\cla}{\spl M(c)\neq0}
\,.
\shortintertext{If $\cx{C}$ is a chain complex of $\gf\cat$-modules, we let}
\minsupp(\cx{C})
&=
\bigcup_{n\in\IZ}\minsupp(C_n)
\,.
\end{align*}
\end{definition}

\begin{remark}
Let $M$ be a $\gf\cat$-module.

(a)
We have $\minsupp(M)\subseteq\supp(M)$, and
$\cl(e)\notin\minsupp(M)$ whenever $\id_e=uw$ for some non-isomorphisms $u,w\in\mor_\cat(e,e)$.

(b)
%Notice that,
If \ $0\TO L\TO M\TO N\TO 0$ \ is a short exact %,
sequence of $\gf\cat$-modules,
then
$
\supp(L)\cup\supp(N)=\supp(M),
$
and,
%\AND
since $\spl$ is right exact,
\[
\minsupp(N) \ \subseteq \
\minsupp(M) \ \subseteq \
\minsupp(N)\cup\minsupp(L).
\]
%\,,
%since $\spl$ is right exact,
However, in general $\minsupp(L)\not\subseteq\minsupp(M)$.

(c)
If $\cat$ is EI, then
any minimal element of~$\supp(M)$ is contained in~$\minsupp(M)$
(compare the proof of \autoref{Nak}).
If in addition $\supp(M)$ is Artinian,
%and $M$ is a $\gf\cat$-module.
then we also have
%any minimal element of~$\supp(M)$ is contained in~$\minsupp(M)$
%(compare the proof of \autoref{Nak}), and we have %that
$
\upset{\minsupp(M)}=\upset{\supp(M)}.
$
\end{remark}

\begin{example}
\label{minsupp-free}
If $F$ is a free $\gf\cat$-module with basis~$B$, then
$\minsupp(F)\subseteq \supp(B)$ and $\supp(F)=\upset{\supp(B)}$.
If $\cat$ is EI, then also $\minsupp(F)=\supp(B)$
%\[
%\minsupp(F)=\SET{\cl(c)\in\cla}{B(c)\neq\emptyset}
%\]
and $\supp(F)=\upset{\minsupp(F)}$.
\end{example}

%%%%%%%%%%%%%%%%%%%%%%%%%%%%%%%%%%%%%%%%%%%%%%%%%%%%%%%%%%%%%%%%%%%%%%%%%%%%%%

\section{Minimal projective resolutions}

We begin by recalling the definitions of projective covers and minimal projective resolutions in an arbitrary abelian category~$\CA$.
An epimorphism $\MOR{f}{L}{M}$ is called \emph{superfluous} if, for any morphism $\MOR{g}{K}{L}$, we have that $g$ is epic if and only if $fg$ is epic.
A \emph{projective cover} of an object~$M$ of~$\CA$ is a projective object~$P$ together with a superfluous epimorphism $\MOR{f}{P}{M}$.
A \emph{minimal projective resolution of~$M$} is a resolution
\[
\MOR{\varepsilon}{\cx{P}}{M}
\quad=\quad
\dotsb\TO P_n\TO[d_{n}]P_{n-1}\TO\dotsb\TO[d_1]P_0\TO[\varepsilon]M
\]
such that $\MOR{\varepsilon}{P_0}{M}$ and $\MOR{d_n}{P_n}{\im(d_n)=\ker(d_{n-1})}$ are projective covers for each~$n\ge1$.

It is well known and easy to see that the definitions imply that projective covers and minimal projective resolutions are unique up to isomorphism, if they exist.
Moreover, if $\MOR{f}{P}{M}$ is a projective cover and $\MOR{g}{Q}{M}$ is an epimorphism with $Q$ projective, then $Q\cong P\oplus P'$ for some $P'\leq\ker{g}$;
if $\MOR{\varepsilon}{\cx{P}}{M}$ is a minimal projective resolution and $\MOR{\delta}{\cx{Q}}{M}$ is any other projective resolution, then $\cx{Q}\cong\cx{P}\oplus\cx{P'}$ for some contractible subcomplex $\cx{P'}\leq\cx{Q}$.

An abelian category~$\CA$ is called \emph{perfect} if every object has a projective cover (and therefore also a minimal projective resolution).
Now suppose that $\CA=\modules{\gf\cat}$, so that we have a notion of finitely generated objects.
Then $\CA$ is called \emph{semi-perfect} if every finitely generated object has a projective cover.
We say that $\CA$ is \emph{Noetherian} if every subobject of a finitely generated object is again finitely generated.
Notice that, if $\CA$ is semi-perfect and Noetherian, then every finitely generated object has a minimal projective resolution.
In the special case where $\CA=\modules{\gf}$, then $\CA$ is perfect, semi-perfect, or Noetherian if and only if the ring~$\gf$ has the same property.

\begin{example}
If $\cat$ is a small category, then the category
\[
\modules{\gf\iso}
\simeq
\prod_{\cl(c)\in\cla}\modules{\gf[c]}
\]
is (semi-)perfect if and only if, for each object $c$ of~$\cat$, the group ring $\gf[c]$ is (semi\nobreakdash-)perfect.
For example, if $\gf$ is a field and all automorphism groups in~$\cat$ are finite, then $\modules{\gf\iso}$ is semi-perfect.
\end{example}

\begin{Nakayama'slemma}
\label{Nak}
Let $\cat$ be a small EI category, and let $M$ be a\/ $\gf\cat$-module.
If\/ $\supp(M)$ has a minimal element, then\/ $\spl M\neq0$.
\end{Nakayama'slemma}

\begin{proof}
This is obvious: if $\cl(c)\in\supp(M)$ is minimal, then there are no non-isomorphisms $b\TO c$ with $M(b)\neq0$, and so by definition $\spl M(c)=M(c)\neq0$.
\end{proof}

\begin{corollary}
\label{cor-Nak}
Let $\cat$ be a small EI category, and let $\MOR{f}{L}{M}$ be a homomorphism of\/ $\gf\cat$-modules.
\begin{enumerate}
\item
\label{i:epi}
Assume that\/ $\supp(M)$ is Artinian.
If\/ $\spl f$ is an epimorphism, then so is~$f$.
\item
\label{i:sup-epi}
Assume that $\supp(L)$ is Artinian and that $f$ is an epimorphism.
If $\spl f$ is a superfluous epimorphism, then so is~$f$.
\end{enumerate}
\end{corollary}

\begin{proof}
\eqref{i:epi}
Consider $\coker f$, and notice that $\supp(\coker f)\subseteq\supp(M)$.
If $f$ is not epic, i.e., if $\supp(\coker f)$ is not empty, then it has a minimal element, and \autoref{Nak} implies that $\spl\coker f\ne0$.
Since $\spl$ is right exact, we have that $\spl\coker f\cong\coker\spl f$ and can conclude that $\spl f$ is not epic.

\eqref{i:sup-epi}
%Since $\supp(M)\subseteq\supp(L)$, we see that $f$ is epic using
%part~\eqref{i:epi}.
Given a homomorphism $\MOR{g}{K}{L}$ such that $fg$ is epic, we get that $\spl(fg)=\spl f\spl g$ is epic, since $\spl$ is right exact, and therefore $\spl g$ is epic since $\spl f$ is superfluous.
Now apply part~\eqref{i:epi} again to conclude that $g$ is epic.
\end{proof}

\begin{lemma}
\label{proj-cov}
Let $\cat$ be a small EI category, and let $M$ be a $\gf\cat$-module.
\begin{enumerate}
\item
\label{i:free-epi}
Assume that $\supp(M)$ is Artinian.
Then $M$ is a quotient of a free $\gf\cat$-module~$P$ with $\minsupp(P)=\minsupp(M)$.
\item
\label{i:proj-cov}
Assume that the upper set generated by~$\supp(M)$ in~$\cla$ is Artinian, and assume that $\spl M$ has a projective cover in $\modules{\gf\iso}$.
Then $M$ has projective cover $P\TO M$ with $\minsupp(P)=\minsupp(M)$.
If in addition projective $\gf\iso$-modules are free, then the projective
cover $P$ of $M$ is free.
\end{enumerate}
\end{lemma}

\begin{proof}
We first describe a general construction.
Let $\MOR{q}{Q}{\spl M}$ be a homomorphism of $\gf\iso$-modules, and assume that $Q$ is projective.
As in \autoref{EI}, write $\ind$ and $\res$ for the induction and restriction functors along the inclusion of categories $\iso\TO\cat$.
Since $Q$ is projective, $q$~factors through the quotient
$\MOR{p_M}{\res M}{\spl M}$ as follows:
\[
\begin{tikzcd}
Q
\arrow[r, "\overline{q}"]
\arrow[rr, "q", rounded corners,
       to path={    (\tikztostart.north)
                 |- +(1.7,.3) [at end]\tikztonodes
                 -| (\tikztotarget.north)}
      ]
&
\res M
\arrow[r, "p_M"]
&
\spl M
\,.
\end{tikzcd}
\]
Inducing from $\iso$ to~$\cat$ we get
\[
\begin{tikzcd}
\ind Q
\arrow[r, "\ind\overline{q}"]
\arrow[dr, "\widetilde{q}"']
&
\ind\res M
\arrow[r, "\ind p_M"]
\arrow[d, "\varepsilon_M"]
&
\ind\spl M
\\
&
M
\end{tikzcd}
\]
where $\varepsilon$ is the counit of the adjunction and we define $\widetilde{q}=\varepsilon_M\circ\ind\overline{q}$.
We claim that $q$ is equal to the composition
\[
\begin{tikzcd}
Q
\arrow[r, "\tau_{Q}"]
&
\spl\ind Q
\arrow[r, "\spl\widetilde{q}"]
&
\spl M
\,,
\end{tikzcd}
\]
where $\tau$ is the natural isomorphism from~\eqref{eq:spl-ind}.
To see that this is true, use the naturality of~$\tau$ and the fact that the composition
\[
\begin{tikzcd}
\res M
\arrow[r, "\tau_{\res M}"]
&
\spl\ind\res M
\arrow[r, "\spl\varepsilon_M"]
&
\spl M
\end{tikzcd}
\]
is equal to~$p_M$; see \autoref{ind-spl}(d).

Since $\tau_{Q}$ is an isomorphism and $q=\spl\widetilde{q}\circ\tau_{Q}$, we see that $\spl\widetilde{q}$ is a (superfluous) epimorphism if and only if $q$ is a (superfluous) epimorphism.

Now, in order to prove~\eqref{i:free-epi}, choose a free $\gf\iso$-module $Q$ together with an epimorphism $\MOR{q}{Q}{\spl M}$ such that $Q(c)\neq0$ if and only if $\spl M(c)\neq0$.
Then $\ind Q$ is free and $\minsupp(\ind Q)=\minsupp(M)$ by construction.
The assumption on~$\supp(M)$ allows us to apply \autoref{cor-Nak}\eqref{i:epi} to conclude that $\MOR{\widetilde{q}}{\ind Q}{M}$ is an epimorphism.

For~\eqref{i:proj-cov}, let $\MOR{q}{Q}{\spl M}$ be a projective cover.
Then $\ind Q$ is projective, and $\supp(\ind Q)=\upset{\minsupp(M)}=\upset{\supp(M)}$.
The assumption on~$\supp(M)$ allows us to apply \autoref{cor-Nak}\eqref{i:sup-epi} to conclude that $\widetilde{q}$ is a projective cover.
\end{proof}

\begin{theorem}[Existence of projective covers and minimal projective resolutions]
\label{min-projres-local}
Let $\cat$ be a small EI~category, and let $M$ be a $\gf\cat$-module.
Assume that the upper set generated by~$\supp(M)$ in~$\cla$ is Artinian.
\begin{enumerate}
\item
\label{i:perfect-local}
If the rings~$\gf[c]$ are perfect for all $\cl(c)\in\upset\minsupp(M)$, then $M$ has a minimal projective resolution. If in addition all projective
$\gf[c]$-modules are free for all $\cl(c)\in\upset\minsupp(M)$ then
the minimal projective resolution of $M$ is a free resolution.
\item
\label{i:semi-perfect-local}
If the rings $\gf[c]$ are semi-perfect for all $\cl(c)\in\upset\minsupp(M)$, and if $M$ is finitely generated, then $M$ has a projective cover.
If additionally the category $\modules{\gf\cat}$ is Noetherian, then $M$ has a minimal projective resolution.
\end{enumerate}
\end{theorem}

\begin{proof}
\eqref{i:perfect-local} is easily proved by applying \autoref{proj-cov}\eqref{i:proj-cov} repeatedly.
The process can be iterated, because, using the notation of the proof of \autoref{proj-cov}, we have that $\supp(\ker\widetilde{q})\subseteq\supp(\ind Q)=\upset{\supp(M)}$ and hence $\upset{\supp(\ker\widetilde{q})}\subseteq\upset{\supp(M)}$.

In order to see that the same strategy also works for~\eqref{i:semi-perfect-local}, recall that the functor~$\spl$ sends finitely generated $\gf\cat$-modules to finitely generated $\gf\iso$-modules.
So the assumptions in~\eqref{i:semi-perfect-local} imply that $\spl M$ has a finitely generated projective cover~$\MOR{p}{P}{\spl M}$.
Therefore $\ind P$ is finitely generated and, if $\modules{\gf\cat}$ is Noetherian, then $\ker\widetilde{p}$ is finitely generated, too.
\end{proof}

Now the following result is an immediate consequence of \autoref{min-projres-local}.

\begin{corollary}
\label{min-projres-global}
Let $\cat$ be a small EI~category.
Assume that every finitely generated upper set in~$\cla$ is Artinian.
\begin{enumerate}
\item
\label{i:perfect-global}
If the category $\modules{\gf\iso}$ is perfect, then any $\gf\cat$-module of finite type has a minimal projective resolution.
\item
\label{i:semi-perfect-global}
If the category $\modules{\gf\iso}$ is semi-perfect, then also $\modules{\gf\cat}$ is semi-perfect.
If additionally $\modules{\gf\cat}$ is Noetherian, then any finitely generated $\gf\cat$-module has a minimal projective resolution.
\end{enumerate}
If in addition all projective $\gf\iso$-modules are free, then the minimal
projective resolutions from parts \eqref{i:perfect-global} and~\eqref{i:semi-perfect-global} above are also free.
\end{corollary}

As a corollary of the proof of \autoref{min-projres-global} we also
have the following

\begin{corollary}
\label{min-res-zero-diff}
Let $\cat$ be a small EI~category, and let $M$ be a $\gf\cat$-module.
Assume that the upper set generated by~$\supp(M)$ in~$\cla$ is Artinian.
Assume that the rings~$\gf[c]$ are semisimple for all $\cl(c)\in\upset\minsupp(M)$, and let $\cx{P}$ be a  minimal projective resolution of $M$.

Then $\spl\cx{P}$ is a chain complex with zero differential.
\end{corollary}

\begin{proof}
As shown in the proof of \autoref{min-projres-local}, the resolution
$\cx{P}$ is constructed at stage $n+1$ by applying $\ind$ to a projective cover of $\spl\ker(P_n\TO P_{n-1})$. Because of the semisimplicity assumption
$\spl\ker(P_n\TO P_{n-1})$ is already projective, therefore for each $n$ we
have that $\spl\bigl( P_{n+1}\TO \ker(P_n\TO P_{n-1})\bigr)$ is an isomorphism.
The desired conclusion is now immediate from the right exactness of $\spl$.
\end{proof}

%%%%%%%%%%%%%%%%%%%%%%%%%%%%%%%%%%%%%%%%%%%%%%%%%%%%%%%%%%%%%%%%%%%%%%%%%%%%%%

\section{Betti categories}
\label{BETTI}

\begin{definition}[Betti category and Betti numbers]
Let $\cat$ be a small EI category.
Suppose that a $\gf\cat$-module~$M$ has a minimal projective resolution $\MOR{\varepsilon}{\cx{P}}{M}$.

(a)
Define the \emph{Betti category of~$M$} to be the full subcategory~$\Betti(M)$ of~$\cat$ with set of objects
\[
\obj\Betti(M)
=
\SET{c\in\obj\cat}{\cl(c)\in\minsupp(\cx{P})}
\,.
\]

(b)
Suppose in addition the minimal projective
resolution $\cx{P}$ is a free resolution. Then, for each object
$c\in\obj\cat$ and for each $n\in\IN$, the $\gf[c]$-module
$\spl P_n(c)$ is free and has a well defined rank, which is equal to
$\#B(c)$, where $B$ is a basis of $P_n$. We set
%number
\[
\beta_{n,c}(M) = \rank_{\gf[c]}\spl P_n(c) = \# B(c)
\,,
\]
and call it the \emph{$n$-th Betti number of $M$ at $c$}.
\end{definition}

Notice that the Betti category of~$M$ and the Betti numbers of $M$
do not depend on~$\cx{P}$, because minimal projective resolutions are unique up to isomorphism.

\begin{example}[Betti category and Betti numbers of the constant module]
\label{betti-constant-functor}
Let $\cat$ be an EI category where all isomorphisms are identities,
and the poset $\cla$ is Artinian and finitely generated. Let $\gf$ be a field
and let $\const_\gf$ be the constant $\gf\cat$-module.
Thus $\supp(\const_\gf)=\cla$ is generated by its finitely many minimal
elements, hence $\const_\gf$ is a finitely generated $\gf\cat$-module.
Since in $\modules{\iso}$ all modules are free,
\autoref{min-projres-global} shows that
$\const_\gf$ has a minimal free resolution, hence a well defined Betti
category and well defined Betti numbers. Since the normalized bar
resolution $\cx{B}(\const_\gf)$ is a free resolution of $\const_\gf$,
it is a direct sum of the minimal free resolution $\cx{P}$ and a
contractible complex.  Furthermore, $\spl\cx{P}$ has zero differential
by \autoref{min-res-zero-diff}. It follows that for each object $c$
we have
\[
\beta_{n,c}(\const_\gf)=\dim_\gf\HH_n\spl\cx{B}(\const_\gf)(c).
\]
Since composition of non-identity morphisms in $\cat$ is not an identity,
the collection $\CN\cat'$ of the non-degenerate simplices of the
nerve of $\cat'$ is a simplicial subcomplex of that nerve
for any subcategory $\cat'$ of $\cat$. In particular this is true
for the categories $\cat_{\le c}$ and $\cat_{< c}$ for all objects $c$ of $\cat$.  Therefore it is straightforward from the definitions that
\[
\beta_{n,c}(\const_\gf)=\dim_\gf\HH_n(\CN\cat_{\le c},\ \CN\cat_{<c};\ \gf ),
\]
the $n$th relative homology of the pair
$(\CN\cat_{\le c},\ \CN\cat_{<c} )$ with coefficients in $\gf$.
In particular, $c$ is an object
of the Betti category of $\const_\gf$ if and only if the pair
$(\CN\cat_{\le c},\ \CN\cat_{<c})$ has nontrivial relative homology over $\gf$.
\end{example}

\begin{example}
\label{projective-twisted-cubic}
Let $\gf$ be a field.
In the polynomial ring~$R=\gf[a,b,c,d]$ with the usual
$\IZ$-grading consider the
homogeneous ideal
\[
I=\langle\,
ac-b^2, bc-ad, bd-c^2
\,\rangle
\,,
\]
the defining ideal of the twisted cubic curve in~$\IP^3$.
%This is the ``Running Example'' in~\cite{Peeva}*{Chapter~IV}.
The $\IZ$-graded minimal free resolution of~$R/I$ over~$R$ is
\[
0
\TO
R^2
\xrightarrow{\begin{psmallmatrix}
d && c \\ c && b \\ b && a
\end{psmallmatrix}}
R^3
\xrightarrow{\begin{psmallmatrix}
ac-b^2 && bc-ad && bd-c^2
\end{psmallmatrix}}
R
\TO
R/I
\TO
0
\,.
\]
The degrees of the basis elements of the free modules
in this resolution are $0, 2$, and $3$, and thus
the Betti category $\Betti$ of the $\IZ$-graded $R$-module $R/I$,
considered  as a functor from the action category
$\act{\IN^4}{\IZ}$ into the category of $\gf$-vector spaces,
is the full subcategory
of~$\act{\IN^4}{\IZ}$ with set of objects $\{0,2,3\}$.
Recall from \autoref{graded-modules-example} that morphisms
in $\act{\IN^4}{\IZ}$ are identified with monomials, in particular
we have
$
\mor_{\Betti}(0,2)=
\mor_{\act{\IN^4}{\IZ}}(0,2)=
\{a^2, ab, ac, ad, b^2, bc, bd, c^2, cd, d^2\}
$,
and
$\mor_{\Betti}(2,3)=\mor_{\act{\IN^4}{\IZ}}(2,3)=\{a,b,c,d\}$. The morphisms
from $0$ to $3$ are given by all monomials of degree $3$.
\end{example}

\begin{example}
In the polynomial ring~$R=\gf[x_{11},x_{12},x_{21},x_{22},x_{31},x_{32}]$
with the standard $\IZ$-grading consider the ideal
\[
I=\langle\,
x_{11}x_{22}-x_{12}x_{21},\; x_{11}x_{32}-x_{12}x_{31},\; x_{21}x_{32}-x_{22}x_{31}
\,\rangle
\]
generated by the $(2\times2)$-minors of the generic $(3\times 2)$-matrix
\[
X=
\begin{pmatrix}
x_{11}&x_{12}
\\
x_{21}&x_{22}
\\
x_{31}&x_{32}
\end{pmatrix}
.
\]
The $\IZ$-graded minimal free resolution of~$R/I$ over~$R$ is
\[
0
\TO
R^2
\TO[X]
R^3
\xrightarrow{(\,
x_{21}x_{32}-x_{22}x_{31}
\ \;
x_{12}x_{31}-x_{11}x_{32}
\ \;
x_{11}x_{22}-x_{12}x_{21}
)}
R
\TO
R/I
\TO
0
\,.
\]
The degrees of the basis elements of the free modules in this resolution are $0$, $2$, and~$3$, and thus the Betti category $\CB$ of the $\IZ$-graded $R$-module $R/I$, considered  as a functor from the action category $\act{\IN^6}{\IZ}$ into the category of $\gf$-vector spaces, is the full subcategory of~$\act{\IN^6}{\IZ}$ with set of objects $\{0,2,3\}$.
Since morphisms in $\act{\IN^6}{\IZ}$ are identified with monomials in $R$, we have for example that
% $
% \mor_{\CB}(0,2)=
% \mor_{\act{\IN^4}{\IZ}}(0,2)=
% \{a^2, ab, ac, ad, b^2, bc, bd, c^2, cd, d^2\}
% $,
% and
$\mor_{\CB}(2,3)=\mor_{\act{\IN^6}{\IZ}}(2,3)=
\{x_{11},x_{12},x_{21},x_{22},x_{31},x_{32}\}$.
% The morphisms from $0$ to $3$ are given by all monomials of degree $3$.
In particular, this Betti category is not equivalent to the Betti category from Example~\ref{projective-twisted-cubic}.
\end{example}

\begin{remark}
The quotient rings from the previous two examples will appear again as instances of toric rings in \autoref{TORIC}; see Examples~\ref{twisted-cubic} and~\ref{generic-toric}.
There the gradings will be different, and the corresponding Betti categories will turn out to be equivalent, and in fact even isomorphic.
\end{remark}

%%%%%%%%%%%%%%%%%%%%%%%%%%%%%%%%%%%%%%%%%%%%%%%%%%%%%%%%%%%%%%%%%%%%%%%%%%%%%%

\section{Main results}

We are now ready to state and prove the general version of our main result, which we specialize to toric rings in the next section.
This result essentially says that the Betti category of a $\gf\cat$-module~$M$ is in a certain sense the ``smallest'' discrete combinatorial object that captures the structure of the minimal projective resolution of~$M$.

\begin{theorem}
\label{main-technical}
Let $\cat$ and~$\catwo$ be small EI~categories.
Assume that:
\begin{itemize}
\item
$\MOR{\varepsilon}{\cx{P}}{M}$
is a minimal projective resolution of a $\gf\cat$-module~$M$;
\item
$\MOR{\delta}{\cx{Q}}{N}$
is a minimal projective resolution of a $\gf\catwo$-module~$N$;
\item
the upper sets\/ $\upset\supp(M)$ and\/ $\upset\supp(N)$ are Artinian;
\item
$\cat'$ is a full subcategory of~$\cat$ such that $\minsupp(\cx{P})\subseteq\cla[\cat']$;
\item
$\catwo'$ is a full subcategory of~$\catwo$ such that $\minsupp(\cx{Q})\subseteq\cla[\catwo']$;
\item
there exist an equivalence of categories $\MOR{\alpha}{\cat'}{\catwo'}$ and an isomorphism
\[
\MOR[\cong]{f}
{\res_{\varrho}(M)}
{\res_{\alpha\varsigma}(N)}
%\,,
\]
of\/ $\gf\cat'$-modules,
where $\cat' \TO[\varrho] \cat$ and $\catwo'\TO[\varsigma]\catwo$ are  the inclusion functors.
\end{itemize}
Then $\alpha$ induces an equivalence of Betti categories
$\Betti(M)\TO \Betti(N)$, and
$f$ induces an isomorphism
\[
\cx{P}
\cong
\ind_{\varrho}\res_{\alpha\varsigma}(\cx{Q})
%\,.
\]
of chain complexes of $\gf\cat$-modules.
\end{theorem}

\begin{remark}
An important special case of this theorem is when $\cat'=\Betti(M)$ and $\catwo'=\Betti(N)$, which satisfy by definition the assumptions on the minimal supports.
Moreover, in many interesting cases (e.g., for monomial ideals and toric rings) the restriction of the corresponding modules to the Betti categories are the constant functors with value the ground field~$\gf$.
In these cases, of course, the existence of the isomorphism~$f$ follows automatically from the existence of the equivalence of categories~$\alpha$.
\end{remark}

% \[
% \begin{tikzcd}[row sep=tiny]
% \Betti(M)
% \arrow[hookrightarrow, r]
% \arrow[dd, "\alpha"']
% &
% \cat
% \arrow[rrd, "M" , pos=.4]
% \\
% &
% &
% &
% \modules{\gf}
% \\
% \Betti(N)
% \arrow[hookrightarrow, r]
% &
% \catwo
% \arrow[rru, "N"', pos=.4]
% \end{tikzcd}
% \]

\begin{remark}
\autoref{main-technical} generalizes (and reproves) \cite{bettiposets}*{Theorem~5.3 on page~5124}.
\end{remark}

For the proof of this theorem, we need the following lemma and corollary.

\begin{lemma}
\label{minsupp-res-ind}
Let $\cat$ be a small EI category and
let $P$ be a projective $\gf\cat$-module with
%Assume that
$\supp(P)$ Artinian.
Let $\cat'$ be a full subcategory of~$\cat$ such that $\minsupp(P)\subseteq\cla[\cat']$, and let
$\MOR{\varrho}{\cat'}{\cat}$ be the inclusion functor.

Then $\res_{\varrho}P$ is a projective $\gf\cat'$-module, and the counit of the adjunction is an isomorphism $\ind_{\varrho}\res_{\varrho}P\TO[\smash{\cong}]P$.
\end{lemma}

\begin{proof}
If $P$ is free with basis~$B$, then $\minsupp(P)=\SET{\cl(c)\in\cla}{B(c)\neq\emptyset}$; compare \autoref{minsupp-free}.
So, if this set is contained in~$\cla[\cat']$, it is clear that $\res_{\varrho}P$ is a free $\gf\cat'$-module and $\ind_{\varrho}\res_{\varrho}P\cong P$.

If $P$ is projective and $\supp(P)$ is Artinian, then $P$ is a direct summand of a free module~$F$ with $\minsupp(F)=\minsupp(P)$ by \autoref{proj-cov}\eqref{i:free-epi}.
Since induction and restriction preserve direct sums, the statements follow.
\end{proof}

\begin{corollary}
\label{min-projres-ind-res}
Let\/ $\cat$ be a small EI~category.
Let $\MOR{\varepsilon}{\cx{P}}{M}$ be a minimal projective resolution of a\/ $\gf\cat$-module~$M$ with\/ $\upset\supp(M)$ Artinian.
Let\/ $\cat'$ be a full subcategory of\/ $\cat$ such that\/ $\minsupp(\cx{P})\subseteq\cla[\cat']$, and let $\MOR{\varrho}{\cat'}{\cat}$ be the inclusion functor.
Then:
\begin{enumerate}
\item
\label{i:ind-res}
$\res_\varrho\cx{P}$ is a projective resolution of\/ $\res_\varrho M$,
and the counit of the adjunction gives isomorphisms\/
$\ind_\varrho\res_\varrho M \cong M$
and\/
$
\ind_\varrho\res_\varrho \cx{P}
\cong
\cx{P}
\,.
$

\item
\label{i:ind}
If $\MOR{\varepsilon'}{\cx{P'}}{\res_\varrho M}$ is a projective resolution in\/ $\modules{\gf\cat'}$,
then\/ $\ind_\varrho\cx{P'}$ is a projective resolution of\/ $M$ in\/ $\modules{\gf\cat}$.
\end{enumerate}
\end{corollary}

\begin{proof}
\eqref{i:ind-res}
Since restriction is an exact functor, $\MOR{\res_{\varrho}\varepsilon}{\res_{\varrho}\cx{P}}{\res_{\varrho}M}$ is a resolution in $\modules{\gf\cat'}$.
Moreover, notice that $\supp(P_n)\subseteq\upset\supp(M)$ for each~$n\ge0$, since by \autoref{proj-cov}\eqref{i:free-epi} the $\gf\cat$-module
$M$ has a free resolution with the same property.
Therefore, the assumptions that $\upset\supp(M)$ is Artinian and $\minsupp(P_n)\subseteq\cla[\cat']$ for each~$n\ge0$ allow us to apply \autoref{minsupp-res-ind} and conclude that $\res_{\varrho}P_n$ is a projective $\gf\cat'$-module, and so $\res_{\varrho}\varepsilon$ is a projective resolution.
Moreover, the counit of the adjunction is an isomorphism $\cx{P}\FROM[\;\smash{\cong}]\ind_{\varrho}\res_{\varrho}\cx{P}$, and hence also $M\cong\ind_{\varrho}\res_{\varrho}M$.

\eqref{i:ind}
By part~\eqref{i:ind-res},
$\MOR{\res_\varrho\varepsilon}{\res_\varrho\cx{P}}{\res_\varrho M}$
is a projective resolution in $\modules{\gf\cat'}$, hence is chain
homotopy equivalent to $\MOR{\varepsilon'}{\cx{P'}}{\res_\varrho M}$.
Therefore $\ind_\varrho\cx{P'}$ is chain homotopy equivalent to
$\ind_\varrho\res_\varrho\cx{P}\cong\cx{P}$. Since induction sends
projectives to projectives, the result follows.
\end{proof}

% \begin{corollary}
% \label{ind-lifts-resolution}
% Let\/ $\cat$ be a small EI~category.
% Let $\MOR{\varepsilon}{\cx{P}}{M}$ be a  minimal
% projective resolution of a $\gf\cat$-module~$M$
% with\/ $\upset\supp(M)$ Artinian, let\/
% $\cat'$ be a full subcategory of\/ $\cat$ such that
% $\minsupp(\cx{P})\subseteq\cla[\cat']$,
% and let $\MOR{\varrho}{\cat'}{\cat}$
% be the inclusion functor.
% Let $\MOR{\varepsilon'}{\cx{P'}}{\res_\varrho M}$ be a
% projective resolution in\/ $\modules{\gf\cat'}$.

% Then\/ $\ind_\varrho\cx{P'}$ is a projective resolution of\/ $M$
% in\/ $\modules{\gf\cat}$.
% %and the counit of the adjunction gives isomorphisms
% %$\ind_\varrho\res_\varrho M \cong M$ and\/
% %$
% %\ind_\varrho\res_\varrho \cx{P}
% %\cong
% %\cx{P}
% %\,.
% %$
% \end{corollary}

\begin{proof}[Proof of \autoref{main-technical}]
Let $\MOR{\varrho}{\cat'}{\cat}$ and $\MOR{\varsigma}{\catwo'}{\catwo}$ be the inclusion functors.
By \autoref{min-projres-ind-res}\eqref{i:ind-res},
%Since restriction is an exact functor,
$\MOR{\res_{\varrho}\varepsilon}{\res_{\varrho}\cx{P}}{\res_{\varrho}M}$ is a projective resolution in $\modules{\gf\cat'}$,
%Moreover, notice that $\supp(P_n)\subseteq\supp(M)$ for each~$n\ge0$, since
%$M$ has a free resolution with the same property.
%Therefore, the assumptions that $\supp(M)$ is Artinian and that
%$\minsupp(P_n)\subseteq\cla[\cat']$ for each~$n\ge0$ allow us to apply
%\autoref{minsupp-res-ind} and conclude that $\res_{\varrho}P_n$ is a
%projective $\gf\cat'$-module, and so $\res_{\varrho}\varepsilon$ is a
%projective resolution.
%Moreover, the counit of the adjunction is an isomorphism
and we have isomorphisms
$\cx{P}\FROM[\;\smash{\cong}]\ind_{\varrho}\res_{\varrho}\cx{P}$ and
%hence also
$M\cong\ind_{\varrho}\res_{\varrho}M$.

The same applies to~$\MOR{\delta}{\cx{Q}}{N}$ and, since $\res_\alpha$ is an equivalence of abelian categories, we get that $\MOR{\res_{\alpha\varsigma}\delta}{\res_{\alpha\varsigma}\cx{Q}}{\res_{\alpha\varsigma}N}$ is a projective resolution in $\modules{\gf\cat'}$.
The isomorphism
$\MOR{f}{\res_{\varrho}M}{\res_{\alpha\varsigma}N}$
therefore lifts to a homotopy equivalence
$\MOR{\cx{f}}{\res_{\varrho}\cx{P}}{\res_{\alpha\varsigma}\cx{Q}}$,
and so using induction we get a homotopy equivalence % of chain complexes in $\modules{\gf\cat}$
$\MOR{\ind_{\varrho}\cx{f}}{\ind_{\varrho}\res_{\varrho}\cx{P}}{\ind_{\varrho}\res_{\alpha\varsigma}\cx{Q}}$.

Since the equivalence of categories~$\alpha$ induces a bijection $\cla[\cat']\cong\cla[\catwo']$, \autoref{min-projres-ind-res}\eqref{i:ind} yields that
%we conclude that each $\ind_{\varrho}\res_{\varsigma\alpha}Q_n$ is a %projective $\gf\cat$-module, and so
$\ind_{\varrho}\res_{\alpha\varsigma}\cx{Q}$ is another projective resolution of~$M$.
But $\cx{P}$ is a minimal projective resolution of~$M$, and so $\ind_{\varrho}\cx{f}$ induces an isomorphism of chain complexes $\cx{P}\oplus\cx{P'}\cong\ind_{\varrho}\res_{\alpha\varsigma}\cx{Q}$ for some contractible subcomplex $\cx{P'}\leq\ind_{\varrho}\res_{\alpha\varsigma}\cx{Q}$.
It follows that $\cx{Q}\cong\ind_{\alpha\varsigma}\res_{\varrho}\ind_{\varrho}\res_{\alpha\varsigma}\cx{Q}\cong\ind_{\alpha\varsigma}\res_{\varrho}\cx{P}\oplus\ind_{\alpha\varsigma}\res_{\varrho}\cx{P'}$, and notice that $\ind_{\alpha\varsigma}\res_{\varrho}\cx{P'}$ is again contractible.
Since $\cx{Q}$ is a minimal projective resolution of~$N$, we conclude that $\ind_{\alpha\varsigma}\res_{\varrho}\cx{P'}\cong 0$, which forces $\cx{P'}\cong 0$, thus %completing the proof of the isomorphism
$
\cx{P}\cong\ind_{\varrho}\res_{\alpha\varsigma}\cx{Q}.
$

%Finally,
The claim that $\alpha$ induces an equivalence
between $\Betti(M)$ and $\Betti(N)$ is now clear since
we have
$
\minsupp(\cx{P}) =
\minsupp(\ind_\varrho\res_{\alpha\varsigma}\cx Q) =
\minsupp(\res_{\alpha\varsigma}\cx Q) \cong
\minsupp(\res_\varsigma\cx Q) =
\minsupp(\cx Q).
$
\end{proof}

%%%%%%%%%%%%%%%%%%%%%%%%%%%%%%%%%%%%%%%%%%%%%%%%%%%%%%%%%%%%%%%%%%%%%%%%%%%%%%

\section{Toric rings}
\label{TORIC}

We now apply the theory developed so far to study homological properties of semigroup rings of the form $\gf[Q]$, where $\gf$ is a field and $Q$ is a \emph{pointed affine semigroup}.
These rings are called \emph{toric rings} because they arise, for example, as affine coordinate rings of toric varieties, and have been studied extensively in geometry, combinatorics, and algebra;
e.g., see \citelist{\cite{Fulton} \cite{Bruns-Gubeladze} \cite{Miller-Sturmfels} \cite{Peeva} \cite{Villarreal}} and the references there.

Recall that a finitely generated abelian semigroup $Q$ is \emph{affine} if it is isomorphic to a subsemigroup of $\IZ^r$ for some positive integer $r$.
The smallest such $r$ is called the \emph{dimension} of $Q$.
The affine semigroup $Q$ is \emph{pointed} if the only invertible element is the identity.
An $r$-dimensional affine pointed semigroup $Q$ has a canonical minimal generating set
$\CM=\{a_1,\dots, a_n\}$, and can be embedded as a subsemigroup of $\IN^r$;
see~\cite{Miller-Sturmfels}*{Corollary 7.23 on page 140}.
% (specifying such an embedding amounts to specifying a torus action on $\gf[Q]$, hence the toric terminology).
Because of this we will always consider the minimal generators $a_i$ as non-zero elements in~$\IN^r$; in particular, no $a_j$ can be expressed as a linear combination of the remaining $a_i$s with nonnegative integer coefficients, and the $(r\times n)$-matrix $A=(a_1\ \dotsb\ a_n)$ with columns the $a_i$s has rank~$r$.
The matrix~$A$ gives a homomorphism of polynomial rings
\[
\MOR{\varphi_A}{\gf[x_1,\dotsc,x_n]}{\gf[t_1,\dotsc,t_r]}
\,,\qquad
x_i\mapsto t_{\vphantom{1}}^{a_{i\vphantom{1}}}=t_1^{a_{i1}}\dotsb t_r^{a_{ir}}
.
\]
with image exactly $\gf[Q]$.
Let $R=\gf[x_1,\dotsc,x_n]$.
Consider the $\IZ^r$-grading on~$R$ given by $\deg x_i=a_i$, and consider the standard $\IZ^r$-grading on $\gf[t_1,\dotsc,t_r]$, i.e., $\deg t_i=e_i$.
Then $\varphi_A$ is graded.
Notice that, while the map $\varphi_A$ and the $\IZ^r$-grading on $R$ depend on the choice of $A$, the ideal $\ker\varphi_A$ depends only on the semigroup~$Q$.
The ideal~$\ker\varphi_A$ is called the \emph{toric ideal} associated with~$Q$ and is denoted~$I_Q$.
With respect to the~$\IZ^r$-grading on $R$ defined above, both $I_Q$ and~$\gf[Q]\cong R/I_Q$ are $\IZ^r$-graded $R$-modules.

Thinking of the matrix~$A$ as a monoid homomorphism $\MOR{A}{\IN^n}{\IZ^r}$, we can form the action category~$\act{\IN^n}{\IZ^r}$ (see \autoref{act-cat-def} and \autoref{act-cat-examples}), and the category of $\IZ^r$-graded $R$-modules is equivalent to the category of $\gf[\act{\IN^n}{\IZ^r}]$-modules by \autoref{mod-over-act-cat}.
As in \autoref{mod-over-act-cat-2} and \autoref{graded-modules-example}, we identify the elements of $\IN^n$ with the corresponding monomials in $R$.
In particular, given $c,d\in\IZ^r=\obj\act{\IN^n}{\IZ^r}$ we have that $\mor_{\act{\IN^n}{\IZ^r}}(c,d)$ is the set of all monomials in $R$ of $\IZ^r$-degree $d-c$, which in general could contain more than one element, and so the category $\act{\IN^n}{\IZ^r}$ is in general not a preorder.
However, the assumptions on the matrix~$A$ imply that all endomorphisms and all isomorphisms in the category $\act{\IN^n}{\IZ^r}$ are identities, and so in particular we have an EI~category, and it is easy to check that for the induced poset structure on $\cla[\act{\IN^n}{\IZ^r}]=\obj\act{\IN^n}{\IZ^r}=\IZ^r$ every finitely generated upper set is Artinian.
Since the semigroup ring $\gf[\IN^n]=R$ is Noetherian, this action category is Noetherian by \autoref{mod-over-act-cat}.
Moreover, each group ring $\gf[c]$ is just $\gf$, hence perfect, and projective $\gf[c]$-modules are free, thus all finitely generated $\gf[\act{\IN^n}{\IZ^r}]$-modules have a minimal free resolution by \autoref{min-projres-global}, and therefore a well defined Betti category and well defined Betti numbers.

We now think of the toric ideal $I_Q$ and the toric ring $\gf[Q]$ also as functors $\act{\IN^n}{\IZ^r}\TO\modules{\gf}$.
Notice that
\[
\supp(\gf[Q])=\im(\MOR{A}{\IN^n}{\IZ^r})=Q = \upset 0
\,.
\]
Moreover, if $\cat$ is any subcategory of~$\act{\IN^n}{\IZ^r}$ with $\obj\cat\subseteq\supp(\gf[Q])=Q$, then
\[
\res_{\cat}\gf[Q] = \const_\gf ,
\]
i.e., the restriction of~$\gf[Q]$ to~$\cat$ is the constant functor~$\gf$.
This applies in particular when~$\cat$ is the Betti category~$\Betti(\gf[Q])$.
Notice that under the equivalence of \autoref{mod-over-act-cat} the objects of $\Betti(\gf[Q])$ are precisely the degrees of the basis elements of the free $\IZ^r$-graded $R$-modules in the minimal $\IZ^r$-graded free resolution of the $\IZ^r$-graded $R$-module~$\gf[Q]$.

\begin{example}[twisted cubic]
\label{twisted-cubic}
Consider the $(2\times 4)$-matrix
\[
A=
\begin{pmatrix}
3&2&1&0\\0&1&2&3
\end{pmatrix}
.
\]
The semigroup $Q$ is the subsemigroup of $\IZ^2$ generated by the columns
of $A$.
The associated toric ideal in~$R=\gf[a,b,c,d]$ is
\[
I_Q=\langle\,
ac-b^2, bc-ad, bd-c^2
\,\rangle
\,,
\]
the defining ideal of the twisted cubic curve in~$\IP^3$;
compare the ``Running Example'' in~\cite{Peeva}*{Chapter~IV}.
The $\IZ^2$-graded minimal free resolution of~$\gf[Q]$ over~$R$ is
\[
0
\TO
R^2
\xrightarrow{\begin{psmallmatrix}
d && c \\ c && b \\ b && a
\end{psmallmatrix}}
R^3
\xrightarrow{\begin{psmallmatrix}
ac-b^2 && bc-ad && bd-c^2
\end{psmallmatrix}}
R
\TO
\gf[Q]
\TO
0
\,.
\]
The $\IZ^2$-degrees of the basis elements of the free modules in this resolution are $(0,0), (2,4), (3,3), (4,2), (4,5)$, and $(5,4)$.
Thus the Betti category of~$\gf[Q]$ is the following full subcategory of~$\act{\IN^2}{\IZ^4}$.
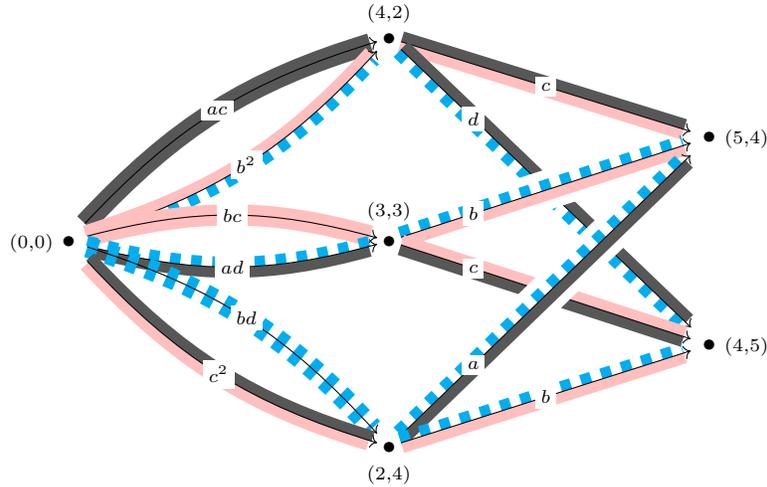
\[
\colorlet{colorX}{black!66}
\colorlet{colorY}{pink}
\colorlet{colorZ}{cyan}
\begin{tikzcd}[column sep=110pt, row sep=25pt, crossing over clearance=15pt]
&
\mathclap{\raisebox{10pt}{$\scriptstyle(4,2)$}}
\mathclap{\bullet}
\arrow[rd,                   -, line width=4pt, shift  left=2pt, shorten >=2.7pt, colorX]
\arrow[rd,                   -, line width=4pt, shift right=2pt, shorten >=2.7pt, colorY]
\arrow[rd,                   "c"            description]
\arrow[rddd,                 -, line width=4pt, shift  left=2pt, shorten >=2.7pt, colorX]
\arrow[rddd,                 -, line width=4pt, shift right=2pt, shorten >=2.7pt, colorZ, dash pattern=on 4pt off 4pt]
\arrow[rddd,                 "d"            description, pos=.25]
\\
&&
\bullet
\mathrlap{\ \scriptstyle(5,4)}
\\
\mathllap{\scriptstyle(0,0)\ }
\bullet
\arrow[ruu, bend  left=15pt, -, line width=8pt,                  shorten >=2.7pt, colorX]
\arrow[ruu, bend  left=15pt, "ac"           description]
\arrow[ruu, bend right=15pt, -, line width=4pt, shift left =2pt, shorten >=2.7pt, colorY]
\arrow[ruu, bend right=15pt, -, line width=4pt, shift right=2pt, shorten >=2.7pt, colorZ, dash pattern=on 4pt off 4pt]
\arrow[ruu, bend right=15pt, "b^2"          description]
\arrow[r,   bend  left=15pt, -, line width=8pt,                  shorten >=2.7pt, colorY]
\arrow[r,   bend  left=15pt, "bc"           description]
\arrow[r,   bend right=15pt, -, line width=4pt, shift left =2pt, shorten >=2.7pt, colorZ, dash pattern=on 4pt off 4pt]
\arrow[r,   bend right=15pt, -, line width=4pt, shift right=2pt, shorten >=2.7pt, colorX]
\arrow[r,   bend right=15pt, "ad"           description]
\arrow[rdd, bend  left=15pt, -, line width=8pt,                  shorten >=2.7pt, colorZ, dash pattern=on 4pt off 4pt]
\arrow[rdd, bend  left=15pt, "bd"           description]
\arrow[rdd, bend right=15pt, -, line width=4pt, shift left =2pt, shorten >=2.7pt, colorX]
\arrow[rdd, bend right=15pt, -, line width=4pt, shift right=2pt, shorten >=2.7pt, colorY]
\arrow[rdd, bend right=15pt, "c^2"          description]
&
\smash{\mathclap{\raisebox{12pt}{$\scriptstyle(3,3)$}}}
\mathclap{\bullet}
\arrow[ru,          phantom, crossing over, shorten >=16pt, shorten <=16pt]
\arrow[ru,                   -, line width=4pt, shift  left=2pt, shorten >=2.7pt, colorZ, dash pattern=on 4pt off 4pt]
\arrow[ru,                   -, line width=4pt, shift right=2pt, shorten >=2.7pt, colorY]
\arrow[ru,                   "b"            description, pos=.25]
\arrow[rd,                   -, line width=4pt, shift  left=2pt, shorten >=2.7pt, colorY]
\arrow[rd,                   -, line width=4pt, shift right=2pt, shorten >=2.7pt, colorX]
\arrow[rd,                   "c"            description, pos=.25]
\\
&&
\bullet
\mathrlap{\ \scriptstyle(4,5)}
\\
&
\mathclap{\raisebox{-10pt}{$\scriptstyle(2,4)$}}
\mathclap{\bullet}
\arrow[ruuu,        phantom, crossing over, shorten >=15pt]
\arrow[ruuu,                 -, line width=4pt, shift  left=2pt, shorten >=2.7pt, colorZ, dash pattern=on 4pt off 4pt]
\arrow[ruuu,                 -, line width=4pt, shift right=2pt, shorten >=2.7pt, colorX]
\arrow[ruuu,                 "a"            description, pos=.25]
\arrow[ru,                   -, line width=4pt, shift  left=2pt, shorten >=2.7pt, colorZ, dash pattern=on 4pt off 4pt]
\arrow[ru,                   -, line width=4pt, shift right=2pt, shorten >=2.7pt, colorY]
\arrow[ru,                   "b"            description]
\end{tikzcd}
\]
The colors are meant to indicate which ``squares'' commute:
two compositions of two morphisms from $(0,0)$ to either~$(5,4)$ or~$(4,5)$ are equal if and only if they are labeled with the same color.
For example, the ``black'' morphisms $(0,0)\TO(5,4)$ given by $c\circ ac$ and $a\circ c^2$ are equal.
\end{example}

Since there are formulas for the $\IZ^r$-graded Betti numbers of the $R$-module $\gf[Q]$ purely in terms of the topological combinatorics of $Q$
(see e.g.~\cite{Miller-Sturmfels}*{Theorem 9.2 on page 192} or \cite{Peeva}*{Chapter IV, Section 67}),
the Betti category of $\gf[Q]$ can be determined without computing the minimal free resolution. This allows us to use it and produce a new canonical (but not minimal) free resolution of $\gf[Q]$.

\begin{theorem}
\label{main-theorem-1}
Let $\gf$ be a field, and let $Q$ be a pointed affine semigroup of rank $r$, with a given embedding into $\IN^r$ and with canonical minimal generators $a_1,\dotsc, a_n$.
Consider $\gf[Q]$ as a $\IZ^r$-graded module over the $\IZ^r$-graded polynomial ring $R=\gf[x_1,\dots, x_n]$, and as a module over the corresponding action category $\act{\IN^n}{\IZ^r}$.
Consider the corresponding Betti category~$\Betti(\gf[Q])$, and let $\MOR{\varrho}{\Betti(\gf[Q])}{\act{\IN^n}{\IZ^r}}$ be the inclusion functor.
Let $\cx{B}(\const_\gf)$ be the normalized bar resolution (\autoref{normalized-bar}) of the constant\/ $\gf\Betti(\gf[Q])$-module.

Then $\ind_\varrho \cx{B}(\const_\gf)$ is a free resolution of\/~$\gf[Q]$ as a module over $\act{\IN^n}{\IZ^r}$.
In particular, applying the equivalence of \autoref{mod-over-act-cat} produces a canonical finite free $\IZ^r$-graded resolution $\cx{F}(\gf[Q])$ of the $\IZ^r$-graded $R$-module\/~$\gf[Q]$.
\end{theorem}

\begin{proof}
Since $\const_\gf=\res_\varrho\gf[Q]$, the chain complex $\ind_\varrho\cx{B}(\const_\gf)$ is a free resolution of $\gf[Q]$ by \autoref{min-projres-ind-res}\eqref{i:ind}.
\end{proof}

\begin{example}
\label{twisted-cubic-bar}
Consider the twisted cubic curve from \autoref{twisted-cubic}.
The nerve of
that Betti category has $12$ non-degenerate $2$-faces,
$18$ non-degenerate $1$-faces, and $6$ (nondegenerate) $0$-faces.
Applying the equivalence of categories from \autoref{mod-over-act-cat}
to the resolution $\ind_\varrho\cx{B}(\const_\gf)$ produces the resolution
\[
0\TO R^{12}\TO[D] R^{18}\TO[E] R^{6}\TO0
\]
of the $\IZ^2$-graded $R$-module $\gf[Q]$, where
\begin{align*}
D&=
\begin{psmallmatrix}
 b & 0 & 0 & 0 & 0 & 0 & a & 0 & 0 & 0 & 0 & 0  \\
 0 & b & 0 & 0 & 0 & 0 & 0 & a & 0 & 0 & 0 & 0  \\
 0 & 0 & c & 0 & 0 & 0 & 0 & 0 & b & 0 & 0 & 0  \\
 0 & 0 & 0 & c & 0 & 0 & 0 & 0 & 0 & b & 0 & 0  \\
 0 & 0 & 0 & 0 & d & 0 & 0 & 0 & 0 & 0 & c & 0  \\
 0 & 0 & 0 & 0 & 0 & d & 0 & 0 & 0 & 0 & 0 & c  \\
 1 & 1 & 0 & 0 & 0 & 0 & 0 & 0 & 0 & 0 & 0 & 0  \\
 0 & 0 & 1 & 1 & 0 & 0 & 0 & 0 & 0 & 0 & 0 & 0  \\
 0 & 0 & 0 & 0 & 1 & 1 & 0 & 0 & 0 & 0 & 0 & 0  \\
-1 & 0 &-1 & 0 & 0 & 0 & 0 & 0 & 0 & 0 & 0 & 0  \\
 0 &-1 & 0 & 0 &-1 & 0 & 0 & 0 & 0 & 0 & 0 & 0  \\
 0 & 0 & 0 &-1 & 0 &-1 & 0 & 0 & 0 & 0 & 0 & 0  \\
 0 & 0 & 0 & 0 & 0 & 0 & 1 & 1 & 0 & 0 & 0 & 0  \\
 0 & 0 & 0 & 0 & 0 & 0 & 0 & 0 & 1 & 1 & 0 & 0  \\
 0 & 0 & 0 & 0 & 0 & 0 & 0 & 0 & 0 & 0 & 1 & 1  \\
 0 & 0 & 0 & 0 & 0 & 0 & 0 & 0 &-1 & 0 &-1 & 0  \\
 0 & 0 & 0 & 0 & 0 & 0 &-1 & 0 & 0 & 0 & 0 &-1  \\
 0 & 0 & 0 & 0 & 0 & 0 & 0 &-1 & 0 &-1 & 0 & 0
\end{psmallmatrix},
\\
E&=
\begin{psmallmatrix}
-c^2 & -bd & -bc & -ad & -b^2  & -ac   &
         0 & 0   & 0   & -bc^2 & -b^2d & -acd &
         0 & 0   & 0   & -b^2c & -ac^2 & -abd     \\
  1  &  1  & 0   & 0   &  0    &  0   &
       -b  & 0   & 0   &  0    &  0   &  0    &
       -a  & 0   & 0   &  0    &  0   &  0                  \\
 0   &  0  & 1   & 1   &  0    &  0   &
        0  & -c  & 0   &  0    &  0   &  0    &
        0  & -b  & 0   &  0    &  0   &  0                  \\
 0   &  0  &  0  & 0   &  1    &  1   &
        0  &  0  & -d  &  0    &  0   &  0    &
        0  &  0  & -c  &  0    &  0   &  0                  \\
 0   &  0  &  0  &  0  &  0    &  0   &
        1  &  1  &  1  &  1    &  1   &  1    &
        0  &  0  &  0  &  0    &  0   &  0                  \\
 0   &  0  &  0  &  0  &  0    &  0   &
        0  &  0  &  0  &  0    &  0   &  0    &
        1  &  1  &  1  &  1    &  1   &  1
\end{psmallmatrix}.
\end{align*}
\end{example}

\begin{remark}
The resolution $\cx{F}(\gf[Q])$ from \autoref{main-theorem-1}
differs from the hull resolution
of $\gf[Q]$ introduced in \cite{Bayer-Sturmfels}. For instance,
the hull resolution of the twisted cubic curve $\gf[Q]$ from
\autoref{twisted-cubic} is a minimal free resolution,
see \cite{Miller-Sturmfels}*{Exercise 9.3 on page 189},
while \autoref{twisted-cubic-bar} shows that $\cx{F}(\gf[Q])$ is
far from minimal. The precise relationship between these two
constructions of canonical
finite free resolutions is not clear at this point.
\end{remark}

Not only does the Betti category of a toric ring determine a canonical free resolution of the ring, but in fact, the equivalence class of the Betti category determines completely the structure of the minimal free resolution of the toric ring in the following sense.

\begin{theorem}\label{main-theorem-2}
Let\/ $\gf$ be a field.
Let $Q$ and $Q'$ be two pointed affine semigroups of ranks $r$ and $r'$ and with Betti categories $\Betti$ and $\Betti'$, respectively.
Denote by
\[
\MOR{\varrho}{\Betti}{\act{\IN^n}{\IZ^r}}
\AND
\MOR{\varrho'}{\Betti'}{\act{\IN^{n'}}{\IZ^{r'}}}
\]
the inclusion functors in the corresponding action categories.
Assume that there is an equivalence of categories
\[
\MOR{\alpha}{\Betti}{\Betti'}
\,,
\]
and let $\cx{P}'$ be a minimal free resolution of the toric ring\/ $\gf[Q']$.
Then the chain complex \/ $\ind_{\varrho}\res_{\alpha\varrho}\cx{P}'$ is a minimal free resolution of the toric ring\/ $\gf[Q]$.
\end{theorem}

\begin{proof}
This is a special case of \autoref{main-technical}.
\end{proof}

% \begin{remark}
% \autoref{main-theorem-2} shows that the Betti category of a toric
% ring is a discrete combinatorial object that plays for toric rings
% the same role as the lcm-lattice \cite{GPW}
% plays for monomial ideals.
% %see \cite{GPW}.
% \end{remark}

\begin{example}[Segre threefold]
\label{generic-toric}
Let $Q$ be the subsemigroup of $\IN^4$ generated by the columns of the $(4\times 6)$-matrix
\[
A=
\begin{pmatrix}
1 & 0 & 0 & 1 & 0 & 0 \\
0 & 1 & 0 & 0 & 1 & 0 \\
0 & 0 & 1 & 0 & 0 & 1 \\
% 1 & 1 & 1 & 0 & 0 & 0 \\
0 & 0 & 0 & 1 & 1 & 1
\end{pmatrix}
.
\]
The associated toric ideal in~$R=\gf[x_{11},x_{12},x_{21},x_{22},x_{31},x_{32}]$ is the ideal
\[
I_Q=\langle\,
x_{11}x_{22}-x_{12}x_{21},\; x_{11}x_{32}-x_{12}x_{31},\; x_{21}x_{32}-x_{22}x_{31}
\,\rangle
\]
generated by the maximal minors of the generic matrix
\[
X=
\begin{pmatrix}
x_{11}&x_{12}
\\
x_{21}&x_{22}
\\
x_{31}&x_{32}
\end{pmatrix}
.
\]
This is the defining ideal of the Segre embedding of~$\IP^2\times\IP^1$ in~$\IP^5$; compare~\cite{Peeva}*{Example~66.10 on pages~264--265}.
It is a \emph{generic determinantal ideal}, in the sense that it specializes via a ring homomorphism from $R$ to any other ideal of maximal minors of a $(3\times 2)$-matrix.
The $\IZ^4$-graded minimal free resolution of~$R/I_Q$ over~$R$ is
\[
0
\TO
R^2
\TO[X]
R^3
\xrightarrow{(\,
x_{21}x_{32}-x_{22}x_{31}
\ \;
x_{12}x_{31}-x_{11}x_{32}
\ \;
x_{11}x_{22}-x_{12}x_{21}
)}
R
\TO
R/I_Q
\TO
0
\,.
\]
The corresponding Betti category is therefore the following.

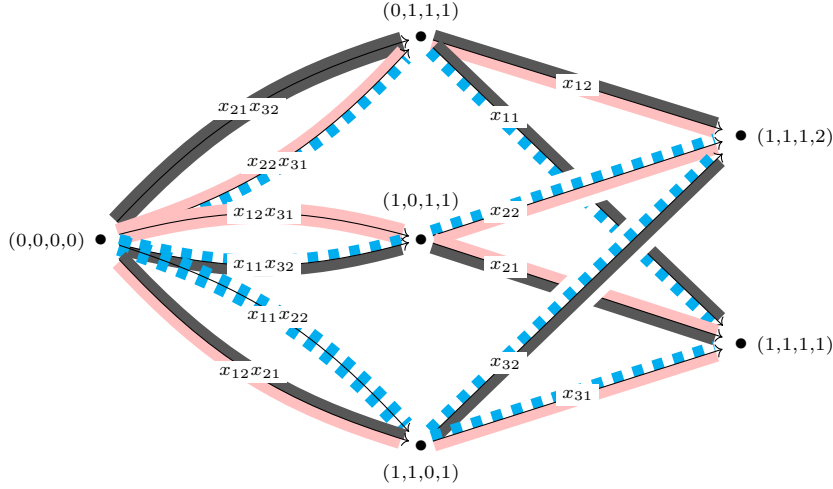
\[
\colorlet{colorX}{black!66}
\colorlet{colorY}{pink}
\colorlet{colorZ}{cyan}
\begin{tikzcd}[column sep=110pt, row sep=25pt, crossing over clearance=15pt]
&
\mathclap{\raisebox{10pt}{$\scriptstyle(0,1,1,1)$}}
\mathclap{\bullet}
\arrow[rd,                   -, line width=4pt, shift  left=2pt, shorten >=2.7pt, colorX]
\arrow[rd,                   -, line width=4pt, shift right=2pt, shorten >=2.7pt, colorY]
\arrow[rd,                   "x_{12}"       description]
\arrow[rddd,                 -, line width=4pt, shift  left=2pt, shorten >=2.7pt, colorX]
\arrow[rddd,                 -, line width=4pt, shift right=2pt, shorten >=2.7pt, colorZ, dash pattern=on 4pt off 4pt]
\arrow[rddd,                 "x_{11}"       description, pos=.25]
\\
&&
\bullet
\mathrlap{\ \scriptstyle(1,1,1,2)}
\\
\mathllap{\scriptstyle(0,0,0,0)\ }
\bullet
\arrow[ruu, bend  left=15pt, -, line width=8pt,                  shorten >=2.7pt, colorX]
\arrow[ruu, bend  left=15pt, "x_{21}x_{32}" description]
\arrow[ruu, bend right=15pt, -, line width=4pt, shift left =2pt, shorten >=2.7pt, colorY]
\arrow[ruu, bend right=15pt, -, line width=4pt, shift right=2pt, shorten >=2.7pt, colorZ, dash pattern=on 4pt off 4pt]
\arrow[ruu, bend right=15pt, "x_{22}x_{31}" description]
\arrow[r,   bend  left=15pt, -, line width=8pt,                  shorten >=2.7pt, colorY]
\arrow[r,   bend  left=15pt, "x_{12}x_{31}" description]
\arrow[r,   bend right=15pt, -, line width=4pt, shift left =2pt, shorten >=2.7pt, colorZ, dash pattern=on 4pt off 4pt]
\arrow[r,   bend right=15pt, -, line width=4pt, shift right=2pt, shorten >=2.7pt, colorX]
\arrow[r,   bend right=15pt, "x_{11}x_{32}" description]
\arrow[rdd, bend  left=15pt, -, line width=8pt,                  shorten >=2.7pt, colorZ, dash pattern=on 4pt off 4pt]
\arrow[rdd, bend  left=15pt, "x_{11}x_{22}" description]
\arrow[rdd, bend right=15pt, -, line width=4pt, shift left =2pt, shorten >=2.7pt, colorX]
\arrow[rdd, bend right=15pt, -, line width=4pt, shift right=2pt, shorten >=2.7pt, colorY]
\arrow[rdd, bend right=15pt, "x_{12}x_{21}" description]
&
\smash{\mathclap{\raisebox{15pt}{$\scriptstyle(1,0,1,1)$}}}
\mathclap{\bullet}
\arrow[ru,          phantom, crossing over, shorten >=16pt, shorten <=16pt]
\arrow[ru,                   -, line width=4pt, shift  left=2pt, shorten >=2.7pt, colorZ, dash pattern=on 4pt off 4pt]
\arrow[ru,                   -, line width=4pt, shift right=2pt, shorten >=2.7pt, colorY]
\arrow[ru,                   "x_{22}"       description, pos=.25]
\arrow[rd,                   -, line width=4pt, shift  left=2pt, shorten >=2.7pt, colorY]
\arrow[rd,                   -, line width=4pt, shift right=2pt, shorten >=2.7pt, colorX]
\arrow[rd,                   "x_{21}"       description, pos=.25]
\\
&&
\bullet
\mathrlap{\ \scriptstyle(1,1,1,1)}
\\
&
\mathclap{\raisebox{-10pt}{$\scriptstyle(1,1,0,1)$}}
\mathclap{\bullet}
\arrow[ruuu,        phantom, crossing over, shorten >=15pt]
\arrow[ruuu,                 -, line width=4pt, shift  left=2pt, shorten >=2.7pt, colorZ, dash pattern=on 4pt off 4pt]
\arrow[ruuu,                 -, line width=4pt, shift right=2pt, shorten >=2.7pt, colorX]
\arrow[ruuu,                 "x_{32}"       description, pos=.25]
\arrow[ru,                   -, line width=4pt, shift  left=2pt, shorten >=2.7pt, colorZ, dash pattern=on 4pt off 4pt]
\arrow[ru,                   -, line width=4pt, shift right=2pt, shorten >=2.7pt, colorY]
\arrow[ru,                   "x_{31}"       description]
\end{tikzcd}
\]
We see that it is equivalent (in fact isomorphic) to the Betti category
of the twisted cubic curve from \autoref{twisted-cubic}.
Therefore, by \autoref{main-theorem-2},
the minimal free resolution of the generic determinantal ideal
can be recovered
in a canonical way from the minimal free resolution of the twisted
cubic ideal, even though there is no map on the level of rings from
$\gf[a,b,c,d]$ to $R$ that can realize that ``specialization''.
In a sense, within the toric world, the defining ideal of the
twisted cubic curve is as generic
as the generic determinantal ideal.
\end{example}

In practice, in order to find the Betti category of a toric ring,
one has to either first compute its minimal free resolution, or
find out the Betti degrees using available formulas for the Betti numbers.
Either way, this involves computing syzygies. We now show how to
replace the Betti category with a slightly bigger canonical category
that can be
directly computed in practice without finding syzygies.

\begin{definition}
Let $Q$ be a pointed affine semigroup of dimension $r$, and let
$\CM=\{a_1,\dots, a_n\}$
be its canonical minimal generating set. Recall that $Q$ is a poset
with $a\le b$ if and only if there is a $c\in Q$ such that $a+c=b$.

(a) The \emph{degree} of a subset $I\subseteq\CM$ is
\[
\deg(I)=\sum_{a\in I} a
\,.
\]
% A set of subsets
% $\{I_1,\dots, I_k\}$ of $\CM$ is called \emph{irredundant} if there are no
% inclusions among the $I_j$s, and $\cap_{j=1}^k I_j=\emptyset$.
% The
% \emph{degree} of the collection $\{I_1,\dots, I_k\}$ is
% $\deg\{I_1,\dots, I_k\} = \sum_{j=1}^k \deg I_j$.

(b) The \emph{least upper bounds category} or \emph{lub-category} of $Q$
is the full subcategory~$\lubcat(Q)$ of the action category $\act{\IN^n}{\IZ^r}$ with
set of objects
\[
\obj\lubcat(Q)=\SET*{a\in Q}{\begin{array}{l}
\!\!\text{there is a set $\{I_1,\dotsc,I_k\}$ of subsets of~$\CM$ such that $a$ is}\!\!
\\[\smallskipamount]
\!\!\text{a least upper bound in $Q$ for the set $\{\deg I_1,\dotsc,\deg I_k\}$}\!\!
\end{array}}
.
\]
\end{definition}

Note that by definition $\lubcat(Q)$ is independent of the choice of the embedding of~$Q$ into $\IN^r$ and of the field $\gf$.

\begin{lemma}
\label{betti-in-lub}
The Betti category of\/ $\gf[Q]$ is a subcategory of $\lubcat(Q)$.
\end{lemma}

\begin{proof}
Let $c$ be an object of the Betti category of $\gf[Q]$.
By \cite{Miller-Sturmfels}*{Theorem 9.2 on page 175}, the reduced homology of the simplicial complex
\[
\Delta_c=\{ I\subseteq \CM \mid \deg I \le c\}
\]
is nontrivial.
Let $\{I_1,\dots, I_k\}$ be the facets of $\Delta_c$.
It is enough to show that $c$ is a least upper bound in $Q$ for the set
$S=\{\deg I_1, \dots, \deg I_k\}$.
Suppose that $b< c$ and $b$ is an upper bound in $Q$ for~$S$.
Then we must have an $a\in\CM$ such that $b<b+a\le c$, and therefore $I_j\cup\{a\}$ is a face of $\Delta_c$ for each~$j$.
Thus $a\in I_j$ for each~$j$, and hence $\Delta_c$ is a cone.
This contradicts the fact that the reduced homology of~$\Delta_c$ is nontrivial.
\end{proof}

Using \autoref{betti-in-lub}, the following theorem is now an immediate consequence of
\autoref{main-technical}, and shows that the lub-category
is a finite combinatorial object that
plays in the toric world a role analogous to that of the
lcm-lattice \cite{GPW} in the world of monomial ideals.

\begin{theorem}
\label{main-theorem-3}
Let $Q$ and $Q'$ be pointed affine semigroups.
Assume that there is an equivalence $\MOR{\alpha}{\lubcat(Q)}{\lubcat(Q')}$
of their lub-categories.
Then, for every field\/ $\gf$, the Betti categories $\Betti(\gf[Q])$ and~$\Betti(\gf[Q'])$ are equivalent.
In particular, the minimal free resolution of\/ $\gf[Q]$ is obtained from the minimal free resolution of\/ $\gf[Q']$ by the functorial procedure described in \autoref{main-theorem-2}.
\end{theorem}

\begin{remark}
It is straightforward to verify that the lub-categories of the pointed affine semigroups from Examples~\ref{twisted-cubic} and~\ref{generic-toric} are not equivalent, even though the corresponding Betti categories are equivalent (even isomorphic) for every field~$\gf$.
\end{remark}

%%%%%%%%%%%%%%%%%%%%%%%%%%%%%%%%%%%%%%%%%%%%%%%%%%%%%%%%%%%%%%%%%%%%%%%%%%%%%%

%%%%%%%%%%%%%%%%%%%%%%%%%%%%%%%%%%%%%%%%%%%%%%%%%%%%%%%%%%%%%%%%%%%%%%%%%%%%%%

\vfill

\end{document}